\documentclass{amsart}

\usepackage[dvips]{graphics}
\usepackage{epsfig}
\usepackage{amssymb}
\usepackage{amscd}
\usepackage[mathscr]{eucal}
\usepackage{enumitem}
\usepackage{url}
\usepackage{color}

\usepackage{hyperref}

\newtheorem{theorem}{Theorem}
\newtheorem{proposition}[theorem]{Proposition}
\newtheorem{corollary}[theorem]{Corollary}
\newtheorem{lemma}[theorem]{Lemma}
\newtheorem{fact}[theorem]{Fact}

\newtheorem{definition}[theorem]{Definition}

\newtheorem{remark}[theorem]{Remark}

\newtheorem{question}[theorem]{Question}

\newtheorem{example}[theorem]{Example}
\newtheorem{rem/def}[theorem]{Remark/Definition}
\newtheorem{notation}[theorem]{Notation}

\numberwithin{theorem}{section}

\providecommand{\Th}{\operatorname{Th}}
\providecommand{\pr}{\operatorname{pr}}
\providecommand{\ch}{\mathop{char}}

\def\CC{{\mathcal C}}

\def\CL{{\mathcal L}}

\def\CR{{\mathcal R}}

\def\BN{{\mathbb N}}
\def\BZ{{\mathbb Z}}
\def\BQ{{\mathbb Q}}
\def\BR{{\mathbb R}}
\def\BC{{\mathbb C}}
\def\BF{{\mathbb F}}
\def\fm{\mathfrak{m}}

\def\dnu{\dot{\nu}}

\def\cfd{K^{\circ}}
\def\cnu{\nu^{\circ}}

\def\Hom{\mathrm{Hom}}

\def\id{\operatorname{Id}}
\def\ov{\overline}
\def\wi{\widetilde}
\def\L{\operatorname{L}}
\def\ob{\operatorname{Ob}}
\def\mor{\operatorname{Mor}}

\def\Th{\operatorname{Th}}

\subjclass[2010]{11U09 (primary)  13L05 (secondary) }


\thanks{The first author was supported by Samsung Science Technology Foundation under Project Number SSTF-BA1301-03. 
The second author was supported by the Yonsei University Research Fund (Post Doc. Researcher Supporting Program) of 2017 (project no: 2017-12-0026).
He was also supported by the National Research Foundation of Korea (NRF) grant funded by the Korea government (MSIT) (No. 2019R1A2C1088609).
\\
\; The authors thank the anonymous referee for valuable comments and suggestions, which were very helpful to reorganize our paper more effectively. 
The authors thank Piotr Kowalski for helpful comments. 
Most of all, the authors thank Thomas Scanlon for detailed and valuable suggestions and comments, which encouraged us to keep writing this article.}

\begin{document}
\setlength{\baselineskip}{12 pt}
\title{On the structure of certain valued fields}

\author[J. Lee]{Junguk Lee}

\address{Instytut Matematyczny, Uniwersytet Wrocławski, pl. Grunwaldzki 2/4, 50-384 Wrocław, Poland
\newline \indent {\em Current address}: Department of Mathematical Sciences, KAIST, 291, Daehak-Ro, Yuseong-Gu, Daejeon, 34141, Republic of Korea}
\email{ljwhayo@kaist.ac.kr}

\author[W. Lee]{Wan Lee}
\address{
Department of Mathematics Yonsei University, 134 Sinchon-Dong, Seodaemun-Gu, Seoul, 120-749, Republic of Korea
\newline \indent {\em Current address}: Department of Mathematical Science, Ulsan National Institute of Science and Technology, Unist-gil 50, Ulsan 44919, Republic of Korea }

\email{wannim@unist.ac.kr}

\renewcommand{\abstractname}{Abstract}

\begin{abstract}
In this article, we study the structure of finitely ramified mixed characteristic valued fields. For any two complete discrete valued fields $K_1$ and $K_2$ of mixed characteristic with perfect residue fields, we show that if the $n$-th residue rings are isomorphic for each $n\ge 1$, then $K_1$ and $K_2$ are isometric and isomorphic.
More generally,
for $n_1\ge 1$,
 there is $n_2$ depending only on the ramification indices of $K_1$ and  $K_2$ such that any homomorphism from the $n_1$-th residue ring of $K_1$ to the $n_2$-th residue ring of $K_2$ can be lifted to a homomorphism between the valuation rings.
Moreover, we get a functor from the category of certain principal Artinian local rings of length $n$ to the category of certain complete discrete valuation rings of mixed characteristic with perfect residue fields, which naturally generalizes the functorial property of unramified complete discrete valuation rings. Our lifting result improves Basarab's relative completeness theorem for finitely ramified henselian valued fields, which solves a question posed by Basarab, in the case of perfect residue fields.
\end{abstract}

  \maketitle
\section{Introduction}
In this paper, we are interested in finitely ramified mixed characteristic valued fields (see Definition \ref{ramfication_index}).
In model theory of valued fields, one of the most important theorems is the AKE-principle, proved by Ax and Kochen in \cite{AK2, AK3}, and independently by Ershov in \cite{E2,E3}. The AKE-principle says that the theory of an unramified henselian valued field of characteristic $0$ is determined by the theory of the residue field and the theory of the value group.
\begin{fact}[The Ax-Kochen-Ershov principle]\label{AKE}\cite{AK2,AK3,E2,E3}
Let $(K_i,k_i,\Gamma_i)$ be an unramified henselian valued field of characteristic zero, where $k_i$ is the residue field and $\Gamma_i$ is the valuation group respectively, for $i=1,2$.
\begin{center}
$K_1\equiv K_2$ if and only if $k_1\equiv k_2$ and $\Gamma_1\equiv \Gamma_2$.
\end{center}
\end{fact}
\noindent Basarab in \cite{B1} generalized the AKE-principle to the finitely ramified case. Actually, he showed that the theory of a finitely ramified henselian valued fields of mixed characteristic is determiend by the theory of each $n$-th residue ring (see Definition \ref{nth_residuering}),  the quotient of the valuation ring by the $n$-th power of the maximal ideal and the theory of the valuation group.
\begin{fact}\label{AKE_Rn}\cite{B1}
Let $(K_i,R_{i,(n)},\Gamma_i)$ be finitely ramified henselian valued fields of mixed characteristic, where $R_{i,(n)}$ is the $n$-th residue ring and $\Gamma_i$ is the valuation group respectively for $i=1,2$. The following are equivalent:
\begin{enumerate}
	\item $K_1\equiv K_2$.
	\item $R_{1,(n)}\equiv R_{2,(n)}$ for each $n\ge 1$ and $\Gamma_1\equiv \Gamma_2$.
\end{enumerate}
\end{fact}

\noindent Motivated by Fact \ref{AKE_Rn}, we ask the following related question on isomorphisms.
\begin{question}\label{motivation_1}
Given two complete discrete valued fields $K_1$ and $K_2$ of mixed characteristic with perfect residue fields, if the $n$-th residue rings of $K_1$ and $K_2$ are isomorphic for each $n\ge 1$, then are $K_1$ and $K_2$ isomorphic\textup{?} Moreover, is there $N>0$ such that $K_1$ and $K_2$ are isomorphic if the $N$-th residue rings of $K_1$ and $K_2$ are isomorphic\textup{?}
\end{question}
\noindent
We give a comment on Question \ref{motivation_1}. Macintyre in \cite{V} raised the following question on the problem of lifting of homomorphisms of the $n$-th residue rings for more general  rings.
\begin{question}\label{modtivation_1'}
Are two complete local noetherian rings $A$ and $B$ isomorphic if the $n$-th residue rings of $A$ and $B$ are isomorphic for each $n\ge 1$\textup{?}
\end{question}
\noindent In \cite{V}, van den Dries gave a positive answer to Question \ref{modtivation_1'} in the case that the residue fields are algebraic over their prime fields. Furthermore, given complete local noetherian rings $A$ and $B$, it is enough to check whether the $N$-th residue rings of $A$ and $B$ are isomorphic for some $N=N(A,B)$ depending on $A$ and $B$. Note that van den Dries showed the existence of a non explicit bound $N$, and in general, there is a counter example by Gabber in \cite{V} for Question \ref{modtivation_1'}.\\

Next we recall the following well-known fact on unramified complete discrete valuation rings.
\begin{fact}\label{Witt_classic}\cite{S}
\begin{enumerate}
\item
Let $k$ be a perfect field of characteristic $p$. Then there exists a  complete discrete valuation ring of characteristic $0$ which is unramified and has $k$ as its residue field. Such a ring is unique up to isomorphism. This unique ring is called the ring of Witt vectors of $k$, denoted by $W(k)$.

\item
Let $R_1 $ and $R_2$ be complete discrete valuation rings of mixed characteristic with perfect residue fields $k_1$ and $k_2$ respectively. Suppose $R_1$ is unramified. Then for every homomorphism $\phi : k_1 \longrightarrow k_2 $, there exists a unique homomorphism $g:R_1 \longrightarrow R_2 $ making the following diagram commutative:
$$
\begin{CD}
R_1 @> {g} >>R_2 \\
@V{\pr_{1,1}} VV @V{\pr_{2,1}} VV\\
k_1 @> \phi >> k_2  ,
\end{CD}
$$
 where two vertical maps are the canonical epimorphisms.
\end{enumerate}
\end{fact}

\noindent In categorical setting, Fact \ref{Witt_classic} is equivalent to the following statement.
\begin{fact}\label{Witt_modern}
Let $\CC_p$ be the category of complete unramified discrete valuation  rings of mixed characteristic $(0,p)$ with perfect residue fields and $\CR_p$ the category of perfect fields of characteristic $p$. Then $\CC_p$ is equivalent to $\CR_p$. More precisely, there is a functor $\L':\ \CR_p\rightarrow\CC_p$ which satisfies:
 \begin{itemize}
 \item
 $\Pr \circ\L'$ is equivalent to the identity functor $\id_{\CR_p} $ where  $\Pr:\CC_p\longrightarrow \CR_p$ is the natural projection functor.
  \item
  $ \L' \circ\Pr$ is equivalent to $\id_{\CC_p}$.
  \end{itemize}
\end{fact}
\noindent Based on Question \ref{motivation_1} and Fact \ref{Witt_modern}, we ask the following generalized questions for the finitely ramified case.

\begin{question}\label{motivation_2}
\begin{enumerate}
\item
For a principal Artinian local ring $\ov{R}$ of length $n$ with a perfect residue field, is there a unique complete discrete valuation ring $R$ which has $\ov{R}$ as its $n$-th residue ring\textup{?} Moreover, if it has a positive answer, can   a lower bound for such $n$ be effectively computed in terms of the ramification index of $\ov{R}$\textup{?}
\item
Given complete discrete valuation rings $R_1$ and $R_2$ of mixed characteristic with perfect residue fields, let $R_{1, (n_1)}$ and $R_{2, (n_2)}$ be the $n_1$-th residue ring of $R_1$ and the $n_2$-th residue ring of $R_2$ respectively.
If $n_1$ and $n_2$ are large enough, is there a unique lifting homomorphism $g:R_1\longrightarrow R_2$ such that $g$ induces a given homomorphism $\phi:R_{1,(n_1)}\longrightarrow R_{2, (n_2)}$\textup{?} Moreover, can such lower bounds on $n_1$ and $n_2$  be effectively computed in terms of the  ramification indices of $R_1$ and $R_2$\textup{?}
\end{enumerate}

\end{question}

\begin{question}\label{motivation_3}
Let $\CC_{p,e}$ be the category of complete discrete valuation rings of mixed characteristic $(0,p)$ with perfect residue fields and ramification index $e$. For $n>e$, let $\CR_{p,e}^n$ be the category of principal Artinian local rings of length $n$ having ramification index $e$ and perfect residue fields (see at the beginning of \emph{Section \ref{Functoriality}} for the precise definition).
Let $\Pr_n:\CC_{p,e}\longrightarrow \CR_{p,e}^n$ be the natural projection functor. Is there a lifting functor $\L:\CR_{p,e}^n\longrightarrow \CC_{p,e}$ which satisfies:
\begin{itemize}
 \item
 $\Pr_n \circ\L$ is equivalent to $\id_{\CR_{p,e}^n} $.
  \item
  $ \L \circ\Pr_n$ is equivalent to $\id_{\CC_{p,e}}$.
  \end{itemize}
\end{question}

\noindent
In general, the answer for Question \ref{motivation_2}.(2) is not positive, that is, there is a homomorphism $\phi:R_{1,n_1} \longrightarrow R_{2,n_2} $ such that no homomorphism from $R_1$ into $ R_2$ induces $\phi$ (see Example \ref{example_tame optimal}).
Instead of finding a `usual' lifting in the sense of Question \ref{motivation_3}, we will show that for sufficiently large $n_2$, if there is a given homomorphism $\phi: R_{1, (n_1)}\longrightarrow R_{2, (n_2)}$, then there is an `approximate' lifting $g:R_1\longrightarrow R_2$ of $\phi$ (see Definition \ref{def:lifting_homomorphim}). \\

Let us come back to the question of elementary equivalence. In \cite{B1}, Basarab posed the following question (see \cite[page 23-24]{B1}):
\begin{question}\label{Basarab_rough_question}
For a finitely ramified henselian valued field $K$ of ramification index $e$, is there a finite integer $N'\ge 1$ depending on $K$ such that any finitely ramified henselian valued field of the same ramification index $e$ is elementarily equivalent to $K$ if their $N'$-th residue rings are elementarily equivalent and their value groups are elementarily equivalent\textup{?}
\end{question}
\noindent Given a finitely ramified henselian valued field $K$, Basarab in \cite{B1} denoted the minimal number $N'$, which satisfies the condition in Question \ref{Basarab_rough_question}, by  $\lambda(T)$ for the complete theory $T$ of $K$.  He showed that $\lambda(T)$ for a local field $K$ is finite but did not give any explicit value of $\lambda(T)$.\\

\medskip

The goal of this paper is to answer these questions when the residue fields are perfect. Its organization is as follows. In Section \ref{preliminaries}, we recall basic definitions and facts.
In Section \ref{LiftHom}, we answer Question \ref{motivation_1} positively for the perfect residue field case in Theorem \ref{main_theorem_construction}.
Our main result shows that if $n_2$ is sufficiently large, then for a given homomorphism $\phi: R_{1, (n_1)}\longrightarrow R_{2, (n_2)}$, there is a homomorphism  $\L(\phi):R_1\longrightarrow R_2$ satisfying a lifting property similar to that of the unramified case.
This provides an answer for Question \ref{motivation_1}. Also, the lifting map $\L$ provides an answer for Question \ref{motivation_2}.(2) and Question \ref{motivation_2}.(1).
In Section \ref{Functoriality}, we concentrate on Question \ref{motivation_3}.
We can show that $\L$ is compatible with the  composition of homomorphisms between residue rings.
More precisely, $\L(\phi_2 \circ \phi_1) =\L(\phi_2)\circ\L(\phi_1)$ for any $\phi_1:R_{1, (n_1)}\longrightarrow R_{2, (n_2)}$ and $\phi_2:R_{2, (n_2)} \longrightarrow R_{3,(n_3)}$.
This defines a functor $\L:\CR_{p,e}^n \longrightarrow \CC_{p,e}$ for sufficiently large $n$.
We prove that a lower bound for $n$ depends only on the ramification index $e$ and the prime number $p$.
Even though $\L$ does not give an equivalence between $\CR_{p,e}^n$ and $\CC_{p,e}$, it turns out that $\L$ satisfies a similar functorial property to that of $\L':\ \CR_p\rightarrow\CC_p$.
This provides an answer for Question \ref{motivation_3}. In Section \ref{AKE_Basarab}, we reduce the problem on elementary equivalence between finitely ramified henselian valued fields of mixed characteristic to the problem on isometricity between complete discrete valued fields of mixed characteristic.
Using results in Section \ref{Crietrion_isom}, we improve Basarab's result on the AKE-principle which gives a positive answer to Question \ref{Basarab_rough_question} when the residue fields are perfect.
Under certain conditions, we calculate $\lambda(T)$ explicitly for the tame case and get a lower bound for $\lambda(T)$ for the wild case.
Surprisingly we show that $\lambda(T)$ can be $1$ even when $K$ is not unramified.
As a special case, we conclude that $\lambda(T)$ is $1$ or $e+1$ if $p\nmid e$, and $\lambda(T)\ge e+1$ if $p \mid e$ when $K$ is a finitely ramified henselian subfield of $\BC_p$ with the ramification index $e$.

\section{Preliminaries}\label{preliminaries}
In this section, we introduce basic notations, terminologies, and several preliminary facts which will be used in this paper. We denote a valued field by a tuple $(K,R,\fm,\nu,k,\Gamma)$ consisting of the following data : $K$ is the underlying field, $R$ is the valuation ring, $\fm$ is the maximal ideal of $R$, $\nu$ is the valuation, $k$ is the residue field, and $\Gamma$ is the value group. Hereafter, the full tuple $(K,R,\fm,\nu,k,\Gamma)$ will be abbreviated in accordance with the situational need for the components. For any field $L$,  $L^{alg}$ denotes a fixed algebraic closure of $L$.

\begin{notation}\label{notation:normalized_valutation}
Let $(L,\nu)$ be a valued field of mixed characteristic $(0,p)$ whose value group is contained in $\mathbb{R}$. We define a normalized valuation $\ov{\nu}$ on $L$ of $\nu$ by the property  $\ov{\nu}(p)=1$, that is, $\nu(p)\ov{\nu}=\nu$. We denote an extended valuation of $\ov{\nu}$ on $L^{alg}$ by $\wi{\nu}$. Note that $\wi{\nu}$ is unique when $L$ is henselian.
\end{notation}

\begin{definition}\label{unramified_valuedfield}
Let $(K,\nu,k,\Gamma)$ be a valued field of characteristic zero. We say $(K,\nu)$ is {\em unramified} if $\ch(k)=0$, or $\ch(k)=p$ and $\nu(p)$ is the minimal positive element in $\Gamma$ for $p>0$. We say $(K,\nu)$ is {\em ramified} if it is not unramified.
\end{definition}

\begin{definition}\label{ramfication_index}
Let $(K,R, \nu,k,\Gamma)$ be a valued field whose residue field has prime characteristic $p$.
\begin{enumerate}
	\item We say $(K,R, \nu,k,\Gamma)$ is {\em finitely ramified} if $K$ is ramified and the set $\{\gamma\in \Gamma|\ 0< \gamma\le \nu(p)\}$ is finite. For $x\in R$, we write $e_{\nu}(x):=|\{\gamma\in \Gamma|\ 0<\gamma\le \nu(x)\}|$. If there is no confusion, we write $e(x)$ for $e_{\nu}(x)$. The number $e_{\nu}(p)$, which is the cardinality of $\{\gamma\in \Gamma|\ 0< \gamma\le \nu(p)\}$, is called {\em the ramification index} of $(K,\nu)$.
	\item Let $(K,R, \nu,k,\Gamma)$ be finitely ramified. If $p$ does not divide $e_{\nu}(p)$, we say $(K,\nu)$ is \emph{tamely ramified}. Otherwise, we say $(K,\nu)$ is \emph{wildly ramified}.
\end{enumerate}
\end{definition}
\noindent Note that if a valued field of mixed characteristic has a finite ramification index, then its value group has a minimum positive element.

\begin{definition}\label{def:totally_ramified}
Let $(R,\nu,k)$ be a complete discrete valuation ring of mixed characteristic with a perfect residue field. Let $(R',\nu',k')$ be a finite extension of $R$. Let $K$ and $K'$ be fraction fields of $R$ and $R'$ respectively. If $k=k'$, we say that $R'$ is a {\em totally ramified extension} of $R$, or $K'$ is a {\em totally ramified extension} of $K$.
\end{definition}

\begin{definition}
Let $(K_1,\nu_1)$ and $(K_2,\nu_2)$ be valued fields. Let $R_1'$ and $R_2'$ be subrings of $K_1$ and $K_2$ respectively. Let $f:R_1'\rightarrow R_2'$ be an injective ring homomorphism. We say $f$ is an {\em isometry} if for $a,b\in R_1'$, $$\nu_1(a)>\nu_1(b)\Leftrightarrow \nu_2(f(a))>\nu_2(f(b)).$$
\end{definition}

\begin{fact}\label{fact:valuationring_homo_isometry}
Let $(R_1,\nu_1)$ and $(R_2,\nu_2)$ be finitely ramified valuation rings of mixed characteristic $(0,p)$ whose value groups are isomorphic to $\BZ$. Let $f:R_1\rightarrow R_2$ be a ring homomorphism. Then we have the following.
\begin{enumerate}
  \item
$f:R_1\rightarrow R_2$ is an isometry.
  \item
Let $K_1$ and $K_2$ be the fraction fields of $R_1$ and $R_2$ respectively.
Then the homomorphism $K_1\longrightarrow K_2$ induced by $f$ is an isometry.
  \item
If both of valuation groups of $R_1$ and $R_2$ are contained in a common ordered abelian group and $\nu_1(p)=\nu_2(p)$, then $\nu_1(x)=\nu_2(f(x))$ for any $x\in R_1$.
\end{enumerate}
\end{fact}
\begin{proof}
(1) We have  $f(n)=n$ for all $n\in \BZ$.
Take $a\in R_1$. Since $f$ sends units to units, $\nu_2(f(a))=0$ if $\nu_1(a)=0$.
To show that $f$ is an isometry, it is enough to show that $\nu_2(f(a))>0$ if $\nu_1(a)>0$.
Suppose $\nu_1(a)>0$. Then there is $k\in R_1^{\times}$ such that $ka^n=p^m$ for some $n,m>0$ since $R_1$ is finitely ramified.
Since $f(p)=p$, we have that $p^m=f(p^m)=f(k)f(a)^n$.
Therefore, we have that $$\nu_1(a)=\frac{m}{n}\nu_1(p),\ \nu_2(f(a))=\frac{m}{n}\nu_2(p)\ (*)$$ and $f$ is injective. Thus, $f$ is an isometry.

(2) This follows directly from (1).

(3) This follows from $(*)$.
\end{proof}

\begin{fact}\label{isometry}
Let $(K_1,\nu_1)$ and $(K_2,\nu_2) $ be valued fields whose value groups are contained in $\BR$.
Let $f:K_1 \longrightarrow K_2$ be an isometry. Suppose $K_1$ is henselian. Let $\widetilde{f} :K_1^{alg}\longrightarrow K_2^{alg}$
be an extended homomorphism of $f$. Then $\widetilde{f}$ is an isometry.
\end{fact}

\begin{proof}
There are two valuations on $\widetilde{f}(K_1^{alg})$, $\widetilde{\nu_1}\circ \widetilde{f}^{-1}$ and $\widetilde{\nu_2}|_{\widetilde{f}(K_1^{alg})}$ where $\widetilde{\nu_2}|_{\widetilde{f}(K_1^{alg})}$ is the restriction of $\widetilde{\nu_2}$ to $\widetilde{f}(K_1^{alg})$.
Since $f$ is an isometry, the restrictions of $\widetilde{\nu_1}\circ \widetilde{f}^{-1}$ and $\widetilde{\nu_2}|_{\widetilde{f}(K_1^{alg})}$ to
$f(K_1)$ are equivalent, in fact, they are equal since  $(\widetilde{\nu_1}\circ \widetilde{f}^{-1})(p)=\widetilde{\nu_2}|_{\widetilde{f} (K_1^{alg})}(p)=1$. Since $K_1$ is henselian,  $f(K_1)$ is Henselian. Hence, $\widetilde{\nu_1}\circ \widetilde{f}^{-1}$ is equal to $\widetilde{\nu_2}|_{\widetilde{f}(K_1^{alg})}$ by the henselian property. This shows that $\widetilde{f}$ is an isometry.
\end{proof}

\begin{definition}\label{nth_residuering}
For a local ring $R$ with maximal ideal $\fm$, we denote $R/\fm^n$ by $R_{(n)}$, and we call $R_{(n)}$ the \emph{$n$-th residue ring} of $R$.
In particular, $R_{(1)}$ is the residue field of $R$.
For each $m>n$, we write $\pr_n:\ R\rightarrow R_{(n)}$ and $\pr^m_n:\ R_{(m)}\rightarrow R_{(n)}$ for the canonical epimorphisms respectively.
\end{definition}
\noindent For $R$-algebras $S_1$ and $S_2$,
we denote the set of $R$-algebra homomorphisms from $S_1$ to $S_2$ by	$\Hom_R(S_1,S_2)$, and we write $\Hom(S_1,S_2)$ for $\Hom_\mathbb{Z}(S_1,S_2)$.

We recall some facts on the structure of finite extensions of unramified complete valued fields.

\begin{fact}\label{witt_embedding}
Let $(R,\nu)$ be a complete discrete valuation ring of mixed characteristic $(0,p)$ with perfect residue field $k$ whose valuation group is $\BZ$. Then
$W(k)$ can be embedded as a subring of $R$ and $R$ is a free $W(k)$-module of rank $\nu(p)$. Moreover, $R$ is a $W(k)$-algebra generated by $\pi$, denoted by $W(k)[\pi]$, where $\pi$ is a uniformizer of $R$.

\end{fact}

\begin{proof}
Chapter 2, Section 5 of \cite{S}
\end{proof}

\begin{fact}\label{Teichmuller}
Let $A$ be a ring that is Hausdorff and complete for a topology defined by a decreasing sequence
$\mathfrak{a}_1 \supset \mathfrak{a}_2 \supset ...$ of ideals such that $\mathfrak{a}_n \cdot \mathfrak{a}_m \subset \mathfrak{a}_{n+m}$.
Assume that the residue ring $A_1=A/\mathfrak{a}_1 $ is a perfect field of characteristic $p$. Then:

\begin{enumerate}
\item
There exists a unique system of representatives $h: A_1 \longrightarrow A$ which commute with $p$-th powers:
$h(\lambda^p ) = h(\lambda)^p $. This system of representatives is called the set of Teichm\"{u}ller representatives.
\item
In order for $ a\in A $ to belong to $S= h(A_1)$, it is necessary and sufficient that $a$ be a $p^n$-th power for all $n\geq 0$.
\item
This system of representatives is multiplicative which means
$$
h( \lambda \mu ) = h(\lambda)h(\mu)
$$
for all $\lambda, \mu \in A_1$.
\item
$S$ contains $0$ and $1$.
 \item
 $S\setminus\{0\}$ is a subgroup of the unit group of $A$.
\end{enumerate}
\end{fact}

\begin{proof}
(1)(2)(3): Chapter 2, Section 4 of \cite{S}

(4): $0$ and $1$ satisfy (2).

(5): (3) and (4) show that $S\setminus\{0\}$ is a subgroup of the unit group of $A$.
\end{proof}

\begin{remark}\label{rem:Teichmuller_nth-residuering}
Let $(R,\fm)$ be a complete discrete valuation ring of mixed characteristic $(0,p)$ with perfect residue field.
By \emph{Fact} \ref{Teichmuller}, $R$ and $R_{(n)}$ have the sets $S$ and $S_n$ of Teichm\"{u}ller representatives respectively. Then, we have that $\pr_n(S)=S_n$.
\end{remark}
\begin{proof}
It is clear that $\pr_n(S)\subset S_n$.
Since each of $S_n$ and $S$ bijectively corresponds to $R/\fm$ by Fact \ref{Teichmuller}, the inclusion must be equality.
\end{proof}

\begin{remark}\label{rem:unique_expresseion_Teichmuller}
Let $(R,\nu)$ be a complete discrete valuation ring of mixed characteristic $(0,p)$ with perfect residue field. Let $S$ be the set of Teichm\"{u}ller representatives and let $\pi$ be a uniformizer. Then, for any $x\in R$, there is a unique infinite sequence $(\lambda_i)_{i\ge 0}$ of elements in $S$ such that $x=\sum_i\lambda_i \pi^i$.
\end{remark}
\begin{proof}
Fix $x\in R$.
By Fact \ref{Teichmuller}, we inductively choose $\lambda_i$'s in $S$ such that $\nu(x-\sum_{i=0}^n \lambda_i\pi^n)>\nu(\pi^n)$ for each $n\ge 0$. Then, we have that $x=\sum_i \lambda_i \pi^i$.
It remains to show that   such  a sequence is unique. Let $(\lambda_i')$ be a sequence of elements in $S$ such that $x=\sum_i \lambda_i'\pi^i$.
Suppose that $\lambda_i\neq \lambda_i'$ for some $i$. Let $i_0$ be the smallest index such that $\lambda_{i_0}\neq\lambda_{i_0}'$. Then, we have that
\begin{align*}
\pr_1\left(\lambda_{i_0}\right)&=\pr_1\left(\frac{x-\sum_{i<i_0}\lambda_i\pi^i}{\pi^{i_0}}\right)\\
&=\pr_1\left(\frac{x-\sum_{i<i_0}\lambda_i'\pi^i}{\pi^{i_0}}\right)\\
&=\pr_1(\lambda_{i_0}'),
\end{align*}
which implies that $\lambda_{i_0}=\lambda_{i_0}'$, a contradiction.
Thus, $(\lambda_i)=(\lambda_i')$.
\end{proof}

\noindent The following facts are useful to effectively compute $N$ in Question \ref{motivation_1} (see Theorem \ref{main_theorem_construction} and Theorem \ref{main_theorem_existence}).
\begin{fact}[Krasner's lemma]\label{Krasner}
Let $(K,\nu)$ be a henselian  valued field and let $a,b\in K^{alg}$. Suppose $a$ is separable over $K(b)$. Suppose that for all embeddings $\sigma(\neq id)$ of $K(a)$ over $K$, we have
$$\wi{\nu}(b-a)>\wi{\nu}\big(\sigma (a)-a\big).
$$
 Then $K(a)\subset K(b)$.
\end{fact}
\begin{proof}
See Chapter 2 of \cite{L} or Theorem 4.1.7 of \cite{EP}.
\end{proof}

\begin{fact}\label{different bound}

Let $(R,\fm_R) \subset (S,\fm_S)$ be discrete valuation rings. Suppose $S=R[\alpha]$ for some $\alpha\in S$ and $S$ is a finitely generated $R$-module so that $\fm_R S=\fm_S^e$ for a positive integer $e$. Suppose the residue fields of $R$ and $S$ are of characteristic $p>0$. Let $f(x)$ in $R[x]$ be a monic irreducible polynomial of $\alpha$ over $R$.

\begin{enumerate}
\item
The different $\mathfrak{D}_{S/R}$ of $S/R$ is a principal ideal generated by $f'(\alpha)$

\item
Let $\nu_{S}$ be the valuation corresponding to $S$. Let $s$ be the power which satisfies $\mathfrak{m}_S^s=\mathfrak{D}_{S/R}$. Then one has
$$\left\{
  \begin{array}{ll}
    s= e-1, & \hbox{if} \  S \ \textrm{is tamely ramified over R, that is, }p\nmid e; \\
    e \le s\le e-1 +\nu_{S}(e), & \textrm{if} \ S \textrm{ is wildly ramified over R, that is, } p\mid e.
  \end{array}
\right.$$
\end{enumerate}

\end{fact}
\begin{proof}
Chapter 3, Section 2 of \cite{N}.
\end{proof}

For model theory of valued fields, we take the language of valued fields with three types of sorts for valuation fields, residue fields, and value groups.
Let $\CL_{K}=\{+,-,\cdot;0,1;|\}$ be a ring language with a binary relation $\mid$ for valued fields, where we interpret the binary relation $\mid$ as $a \mid b$ if $\nu(a)\le \nu(b)$ for $a,b\in K$, $\CL_{k}=\{+',-',\cdot';0',1'\}$ be the ring language for residue fields, and $\CL_{\Gamma}=\{+^*;0^*;<\}$ be the ordered group language for valuation groups.
The language of valued fields is the language $\CL_{val}=\CL_{K}\cup\CL_{k}\cup\CL_{\Gamma}$ equipped with function symbols $\pr_{k}$ and $\pr_{\Gamma}$, where $\pr_k$ and $\pr_{\Gamma}$ are interpreted as the canonical surjctive maps from the valuation ring to the residue field and from the valued field to the valuation group respectively.
Next, we consider an extended language of $\CL_{val}$ by adding the ring languages for the $n$-th residue rings and function symbols $\pr_n$ and $\pr_m^n$ for $n\ge m$, where $\pr_n$ and $\pr_m^n$ are interpreted as the canonical epimorphisms from the valuation ring to the $n$-th residue ring and from the $n$-th residue ring to the $m$-th residue ring respectively.
For each $n\ge 1$, let $\CL_{R_{(n)}}=\{+_n,-_n,\cdot_n;0_n,1_n\}$ be the ring language for the $n$-th residue ring. For $n=1$, we identify $\CL_{R_{(1)}}=\CL_{k}$.
We get an extended language $\CL_{val,R}=\CL_{val}\cup\bigcup_{n\ge 1} \CL_{R_{(n)}}$ for valued fields.
Let $(K_1,\nu_1,k_1,\Gamma_1)$ and $(K_2,\nu_2,k_2,\Gamma_2)$ be valued fields, and let $R_{1,{(n)}}$ and $R_{2,{(n)}}$ be the $n$-th residue rings of $(K_1,\nu_1)$ and $(K_2,\nu_2)$ respectively.
We say $(K_1,\nu_1)$ and $(K_2,\nu_2)$ are elementarily equivalent if they are elementarily equivalent in $\CL_{K}$.
If $(K_1,\nu_1)$ and $(K_2,\nu_2)$ are elementarily equivalent, then
\begin{itemize}
	\item $k_1$ and $k_2$ are elementarily equivalent in $\CL_{k}$;
	\item $\Gamma_1$ and $\Gamma_2$ are elementarily equivalent in $\CL_{\Gamma}$; and
	\item $R_{1,{(n)}}$ and $R_{2,{(n)}}$ are elementarily equivalent in $\CL_{R_{(n)}}$ for each $n\ge 1$.
\end{itemize}

\begin{remark}\label{iso_Rn_Gamma}
Let $(K_1,\nu_1,\Gamma_1)$ and $(K_2,\nu_2,\Gamma_2)$ be valued fields. Suppose
\begin{itemize}
	\item $R_{1,{(n)}}\equiv R_{2,{(n)}}$ as rings in the language $\CL_{R_{(n)}}$ for each $n\ge 1$;
	\item $\Gamma_1\equiv \Gamma_2$ as ordered abelian groups in the language $\CL_{\Gamma}$.
\end{itemize}
Then there are $\aleph_1$-saturated elementary extensions $(K_1',\nu_1',\Gamma_1')$ and $(K_2',\nu_2',\Gamma_2')$ of $K_1$ and $K_2$ such that
\begin{itemize}
	\item $R_{1,{(n)}}'\cong R_{2,{(n)}}'$ for $n\ge 1$;
	\item $\Gamma_1'\cong \Gamma_2'$,
\end{itemize}
 where $R_{1,{(n)}}'$ and $R_{2,{(n)}}'$ are the $n$-th residue rings of $K_1'$ and $K_2'$ respectively.
\end{remark}

\begin{proof}
It is easily deduced from the the Keisler-Shelah isomorphism theorem.
\end{proof}

Next, we review coarse valuations. For the coarse valuations, we refer to \cite{K, PR}.

\begin{rem/def}\label{def_coarse_valuation}\cite[page 25-27]{PR}
Suppose $(K,\nu,k,\Gamma)$ is finitely ramified. Let $\pi$ be a uniformizer so that $\nu(\pi)$ is the smallest positive element in $\Gamma$. Let $\Gamma^{\circ}$ be the convex subgroup of $\Gamma$ generated by $\nu(\pi)$ and $\dot{\nu}:K\setminus\{0\}\longrightarrow \Gamma/\Gamma^{\circ}$ be a map sending $x(\neq 0)\in K$ to $\nu(x)+\Gamma^{\circ}\in \Gamma/\Gamma^{\circ}$. The map $\dot{\nu}$ is a valuation, called {\em the coarse valuation}. The residue field $\cfd$ of $(K,\dot{\nu})$, called {\em the core field} of $(K,\nu)$,  forms a valued field equipped with a valuation $\cnu$, whose value group is $\Gamma^{\circ}$. More precisely, the valuation $\cnu$ is defined as follows: Let $\pr_{\dnu} : R_{\dnu}\longrightarrow \cfd$ be the canonical epimorphism and let $x\in R_{\dnu}$. If $x^{\circ}:=\pr_{\dnu}(x)\in \cfd\setminus\{0\}$, then $\cnu(x^{\circ}):=\nu(x)$. And $x^{\circ}=0\in \cfd$ if and only if $\nu(x)>\gamma$ for all $\gamma\in \Gamma^{\circ}$.
\end{rem/def}

\begin{remark}\label{coarse_valuation}
\begin{enumerate}
	\item Let $R_{\nu}$, $R_{\dnu}$, and $R_{\cnu}$ be the valuation rings of $(K,\nu)$, $(K,\dnu)$, and $(\cfd,\cnu)$ respectively. Then $(\pr_{\dnu})^{-1}(R_{\cnu})=R_{\nu}$.
	\item Let $R_{(n)}$ and $R^{\circ}_{(n)}$ be the $n$-th residue rings of $(K,\nu)$ and $(\cfd,\cnu)$ respectively. Then there is a canonical isomorphism $\theta_n:R_{(n)}\longrightarrow R^{\circ}_{(n)}$ such that $\pr_n^{\cnu}\circ(\pr_{\dnu}|_{R_{\nu}})=\theta_n\circ\pr_n$, where $\pr_n:R_{\nu}\longrightarrow R_{(n)}$ and $\pr_n^{\cnu}:R_{\cnu}\longrightarrow R^{\circ}_{(n)}$ are the canonical epimorphisms.
	\item If $(K,\nu)$ is henselian, then $(K,\dnu)$ is henselian.
	\item If $(K,\nu)$ is $\aleph_1$-saturated, then $(\cfd,\cnu)$ is complete.
\end{enumerate}
\end{remark}
\begin{proof}
(1) Note that $R_{\dnu}:=\{x\in K|\ \dnu(x)\ge 0\}=\{x\in K|\ \nu(x)\ge \gamma\ \mbox{for some } \gamma\in \Gamma^{\circ} \}$.
Let $x\in R_{\dnu}$ be such that $\pr_{\dnu}(x)=:x^{\circ} \in R_{\cnu}$, that is, $\cnu(x^{\circ})(\in \Gamma^{\circ})\ge 0$.
If $x^{\circ}=0$, $\nu(x)>\gamma$ for all $\gamma\in \Gamma^{\circ}$ and $x\in R_{\nu}$.
If $x^{\circ}\neq 0$, then $\cnu(x^{\circ})=\nu(x)\ge 0$ in $\Gamma^{\circ}$, and hence $\nu(x)\ge 0$ in $\Gamma$.
Thus $x\in R_{\nu}$.
Therefore, for $x\in R_{\dnu}$, $x\in R_{\nu}$ if and only if $x^{\circ}\in R_{\cnu}$.

(2) Note that each $\theta_n$ is induced from $\pr_{\dnu}|_{R_{\nu}}:R_{\nu}\longrightarrow R_{\cnu}$.
It is easy to see that each $\theta_n$ is surjective.
To show that $\theta_n$ is injective, it is enough to show that $\nu(x)\ge n$ if and only if $\cnu(x^{\circ})\ge n$ for $x\in R_{\nu}$. It clearly comes from the definition of $\cnu$ in (1).

(3)-(4) Section 5 of \cite{K}.
\end{proof}

\begin{remark}\label{rem:reduction_elementary_to_isomorphism}
By combining \emph{Fact \ref{AKE}}, \emph{Remark \ref{iso_Rn_Gamma}} and \emph{Remark \ref{coarse_valuation}},  we reduce the problem on elementary equivalence between finitely ramified henselian valued fields of mixed characteristic to the problem on isometricity between complete discrete valued fields of mixed characteristic whose  $n$-th residue rings are isomorphic for each $n\ge 1$. To our knowledge, this strategy first appeared in \cite{K}.
\end{remark}

\section{Lifting homomorphisms}\label{Crietrion_isom}\label{LiftHom}
{\bf From now on, if there is no comment, we consider only complete discrete valued fields of mixed characteristic $(0,p)$ with perfect residue fields, and we assume that valuation groups are $\BZ$ so that for a valued field $(L,R,\nu)$, $\nu(x)=e_{\nu}(x)$ for $x\in R$.}
Let $(R,\nu,k)$ be a valuation ring.
Let $\pi$ be a uniformizer of $R$.
Let $L$ and $K$ be the fraction fields of $R$ and $W(k)$ respectively.

\begin{definition}\label{def:M(R)}
If $L$ is {\bf ramified}, we denote the maximal value
$$
 \max\left\{\wi{\nu}\left(\pi- \sigma(\pi) \right): \sigma\in \Hom_{K}\left(L,L^{alg}\right),\ \sigma (\pi) \neq \pi \right\}
$$
by $M(R)_{\pi}$ or $M(L)_{\pi}$.
\end{definition}

\begin{lemma}\label{restriction map}
Let $(R_i,\fm_i,\nu_i,k_i)$ be a valuation ring and let $\pi_i$ be a uniformizer of $R_i$ for $i=1,2$. Let $S_i$ be the set of Teichm\"uller representatives of $R_i$ for $i=1,2$.
\begin{enumerate}
 \item
 For any homomorphism
 $\phi:R_{1,(n_1)}\longrightarrow R_{2,(n_2)}$, $\phi (S_1 + \fm_1 ^{n_1} )$ is contained in $S_2 +\fm_2^{n_2} $.
 Similarly, for any homomorphism $g:R_1\longrightarrow R_2$, $g(S_1)$ is contained in $S_2$.

 \item
 For any homomorphism $\phi:R_{1,(n_1)}\longrightarrow R_{2,(n_2)}$, $\phi\left((W(k_1)+\fm_1^{n_1})/\fm_1^{n_1}\right)$ is contained in $ (W(k_2)+\fm_2^{n_2})/\fm_2^{n_2}$.
 Similarly, for any homomorphism $g:R_1\longrightarrow R_2$, $g(W(k_1))$ is contained in $W(k_2)$.
\end{enumerate}
\end{lemma}

\begin{proof}
(1)
This comes from Fact \ref{Teichmuller} and Remark \ref{rem:Teichmuller_nth-residuering}.

\smallskip

(2) Since $W(k_i) /pW(k_i)\cong R_i/\fm_i \cong k_i$, $S_i$ is contained in $W(k_i)$ by Fact \ref{Teichmuller}.
Since any element $a$ in $W(k_1)$ can be uniquely written  as $a=\sum_{r=0}^{\infty} \lambda_r p^r$ where $\lambda_r$ is in $S_1$, we have that $\phi\left((W(k_1)+\fm_1^{n_1})/\fm_1^{n_1}\right)\subset  (W(k_2)+\fm_2^{n_2})/\fm_2^{n_2}$ and $g(W(k_1))\subset W(k_2)$ by Lemma \ref{restriction map}.(1).
\end{proof}

\begin{lemma}\label{M_1=M_2}
Let $L_i$ and $K_i$ be the fraction fields of  $R_i$ and $W(k_i)$ respectively for $i=1,2$.
\begin{enumerate}
\item
Let $\alpha$ be a uniformizer of $R_1$.
Then $M(R_1)_{\pi_1}=M(R_1)_{\alpha}$.
We write $M(R_1)_{\pi_1}=M(R_1)$.
\item
Suppose $[L_1:K_1]=[L_2:K_2]=e$, that is, $\nu_1(p)=\nu_2(p)=e$. Suppose there is an isometry $g: L_1\longrightarrow L_2$.
Then $M(R_1)=M(R_2)$.

\end{enumerate}
\end{lemma}

\begin{proof}
(1)
By Remark \ref{rem:unique_expresseion_Teichmuller}, we can write $\alpha= \sum_{r=1}^{\infty}\lambda_r \pi_1^r$
where $\lambda_r$ is a   Teichm\"{u}ller representative of $R_1$ for each $r$ and $\lambda_1 \neq0 $.
Since $R_1 /\fm_1 = k_1$, $\lambda_r$ is in $W(k_1)$ for each $r$ by Fact \ref{Teichmuller}. For any $\sigma$ in $\Hom_{K_1}(L_1,K_1^{alg})$,
\begin{align*}
\alpha-\sigma(\alpha)&=\sum_{r=1}^{\infty}\lambda_r \pi_1^r-\sigma\left(\sum_{r=1}^{\infty}\lambda_r \pi_1^r\right)\\
&=\sum_{r=1}^{\infty}\lambda_r\Big( \pi_1^r-\sigma(\pi_1^r)\Big)\\
&=\Big(\pi_1-\sigma(\pi_1)\Big)\sum_{r=1}^{\infty}\lambda_r
\left(\sum_{j=0}^{r-1}\pi_1^{r-1-j} \sigma(\pi_1^j)\right)
\end{align*}
and
$
\wi{\nu_{1}}(\alpha-\sigma(\alpha))=\wi{\nu_1}(\pi_1- \sigma(\pi_1))$ because
$$
\wi{\nu_1}\left(\sum_{r=1}^{\infty}\lambda_r\left(\sum_{j=0}^{r-1}\pi_1^{r-1-j}  \sigma(\pi_1^j)\right)\right)=0.
$$

\noindent
So, we have $M(R_1)_{\pi_1}=M(R_1)_{\alpha} $.

\smallskip

(2)
By Lemma \ref{restriction map}.(2), $g(K_1)$ is contained in $K_2$.
Let $f_1$ be the monic irreducible polynomial of $\pi_1$ over $W(k_1)$.
Since $g$ is an isometry, we have $\ov{\nu_2}(g(\pi_1))=\ov{\nu_1}(\pi_1)= 1/e$, and hence, $g(\pi_1)$ is a uniformizer of $L_2$. Let $\wi{g}:L_1^{alg}\longrightarrow L_2^{alg}$ be an extended homomorphism of $g$. If we write $f_1=x^e + \cdots +a_1x +a_0$, we have that
\begin{align*}
g(f_1)&=x^e+ \cdots+g(a_1)x+g(a_0)
\end{align*}
is the monic irreducible polynomial of $g(\pi_1)$ over $K_2$ since $g(K_1)$ is contained in $K_2$. Then by Lemma \ref{M_1=M_2}.(1) and Fact \ref{isometry}, we get
\begin{align*}
M(R_2)&= \max\left\{\wi{\nu_2}\left(g(\pi_1)-\eta \right):g(f_1)(\eta)=0 ,\ \eta \neq g(\pi_1) \right\}\\
&=\max\left\{\wi{\nu_2}\left(g(\pi_1)- \wi{g}(\pi_1') \right):f_1(\pi_1')=0 ,\ \pi_1' \neq \pi_1 \right\}\\
&=\max\left\{\wi{\nu_1}\left(\pi_1- \pi_1' \right):f_1(\pi_1')=0 ,\ \pi_1' \neq \pi_1 \right\}\\
&=M(R_1),
\end{align*}
which finishes the proof.
\end{proof}

Now we introduce the notion of lifting maps.
\begin{definition}\label{def:lifting_homomorphim}
Let $R_1 $ and $R_2$ be complete discrete valuation rings of characteristic $0$ with perfect residue fields $k_1$ and $k_2$ of characteristic $p$ respectively. Let $\mathfrak{m}_i$ be the maximal ideal of $R_i$ for $i=1,2$.
Let $L_i$ and $K_i$ be the fraction fields of $R_i$ and $W(k_i)$ for $i=1,2$ respectively. For any homomorphism $\phi:R_{1,(n_1)}\longrightarrow R_{2,(n_2)}$, we say that  a homomorphism $g:R_1\longrightarrow R_2$  is a \emph{$(n_1,n_2)$-lifting} of $\phi$ if $g$ satisfies the following:
 \begin{itemize}
\item
For any $x$ in $R_1$, there exists a representative $\beta_x$ of $\phi(x+\fm_1^{n_1})$ which satisfies
$$
  \widetilde{\nu_2}\big(g(x)-\beta_x\big)
  > M(R_1)
  $$
\item
$\phi_{red,1} \circ \pr_{1,1}=\pr_{2,1} \circ g$ where $\phi_{red,1}:k_1 \longrightarrow k_2 $ denotes the natural reduction map of $\phi$ and $\pr_{i,1}:R_i \longrightarrow k_i$ is the canonical epimorphism for $i=1,2$.
\end{itemize}
When such $g$ is unique, we denote $g$ by $\L_{n_1,n_2}(\phi)$. When $\L_{n_1,n_2}(\phi)$ exists for all $\phi:R_{1,(n_1)}\longrightarrow R_{2,(n_2)}$, we write $\L_{n_1,n_2}:\Hom(R_{1,(n_1)},R_{2,(n_2)})\longrightarrow \Hom(R_1,R_2)$. When $n_1=n_2=n$, we denote $\L_{n_1,n_2}$ by $\L_{n}$ and say that $\L_{n}$ is an \emph{$n$-lifting}.
\end{definition}

\noindent The following example explains why we need our `approximate' lifting map for the ramified case.
\begin{example}\label{example_tame optimal}
If we take $R_1 = R_2 = \mathbb{Z}_3 [\sqrt3]$ and $n_1=n_2=2n$, then $R_{1,(2n)} =R_{2,(2n)} \cong (\mathbb{Z}_3 /3^{n}\mathbb{Z}_3) [x]/ (x^2 -3) $.
Then $\phi: a+bx \mapsto a+(1+3^{n-1})bx =\phi(a+bx)$ defines an isomorphism between $R_{1,(2n)} $ and $R_{2,(2n)} $.
But when $n>1$, there is no homomorphism $g:R_1 \longrightarrow R_2$ which induces $\phi$ since the Galois conjugates of $\sqrt{3}$ are $\pm \sqrt{3}$.
This shows that we can not guarantee that the following diagram is commutative:
$$
\begin{CD}
R_1 @>\L_{n_1,n_2}(\phi)>> R_2\\
@V VV @V VV \\
R_{1,(n_1)} @>\phi>>R_{2,(n_2)}
\end{CD}
$$
\end{example}

We introduce a weaker condition of lifting map, which will turn out to be equivalent to Definition \ref{def:lifting_homomorphim} (see Proposition \ref{proposition_lifting homomorphism}). This weaker notion is useful to show  the functoriality of lifting maps (see Proposition \ref{proposition lifting functor}).

\begin{proposition}\label{proposition_lifting homomorphism}
For a homomorphism $\phi:R_{1,(n_1)}\longrightarrow R_{2,(n_2)}$, suppose that  a homomorphism $g:R_1\longrightarrow R_2$ satisfies the following:
 \begin{itemize}
\item
There exists a representaive $\beta$ of $\phi(\pi_1 +\fm_1^{n_1})$ which satisfies
$$
  \widetilde{\nu_2}\left(g(\pi_1)-\beta\right)
  > \max_{\sigma}\limits \left\{\widetilde{\nu_2}\left(\sigma\left(g(\pi_1)\right) -\beta  \right):\sigma\left(g(\pi_1)\right)\neq g(\pi_1)   \right\}
$$
          where $\sigma$ runs through all of $\Hom_{K_2}(L_2,L_2^{alg})$.
\item
$\phi_{red,1} \circ \pr_{1,1}=\pr_{2,1} \circ g$ where $\phi_{red,1}:k_1 \longrightarrow k_2 $ is the natural reduction map of $\phi$.
\end{itemize}

\begin{enumerate}
\item
We have that
$$ \max_{\sigma}\limits\left\{\widetilde{\nu_2}\left(\sigma\left(g(\pi_1)\right) -\beta  \right):\sigma\left(g(\pi_1)\right)\neq g(\pi_1)\right\}=M(R_1).
$$

\item

For any $x$ in $R_1$, there exists a representative $\beta_x$ of $\phi(x+\fm_1^{n_1})$ which satisfies
$$
  \widetilde{\nu_2}\left(g(x)-\beta_x\right)
  > M(R_1)
  $$
so that $g$ is a $(n_1, n_2)$-lifting of $\phi$.
\end{enumerate}

\end{proposition}

\begin{proof}

(1)
For $\sigma\in \Hom_{K_2}(L_2, L_2 ^{alg})$ with $\sigma ( g(\pi_1) )\neq g(\pi_1)$, we have
\begin{align*}
\widetilde{\nu_2}\left(\sigma\left(g(\pi_1)\right)-g(\pi_1)\right) &= \widetilde{\nu_2}\left(\sigma\left(g(\pi_1)\right)-\beta +\beta-g(\pi_1) \right) \\
&=\min\left\{\widetilde{\nu_2}\left(\sigma\left(g(\pi_1)\right) -  \beta  \right),\
\widetilde{\nu_2}\left(g(\pi_1)-\beta \right) \right\}\\
&=\widetilde{\nu_2}\left(\sigma\left(g(\pi_1)\right) -  \beta  \right)
\end{align*}
where the second equality follows from the ultrametric inequality and the assumption $\widetilde{\nu_2}(g(\pi_1)-\beta) >\widetilde{\nu_2}(\sigma(g(\pi_1))-\beta)$.

This shows
\begin{align*}
M(R_1)&=\max_{\sigma'}\left\{\wi{\nu_1}\left(\pi_1- \sigma'(\pi_1) \right): \sigma' (\pi_1) \neq \pi_1 \right\}\\
&=\max_{\sigma}\left\{\widetilde{\nu_2}\left(g(\pi_1)- \sigma\left( g(\pi_1)\right)  \right) :\sigma\left(g(\pi_1)\right)\neq g(\pi_1) \right\} \\ &=\max_{\sigma}\left\{\widetilde{\nu_2}\left(\sigma\left(g(\pi_1)\right) -\beta  \right):\sigma\left(g(\pi_1)\right)\neq g(\pi_1) \right\}
\end{align*}

\noindent
where $\sigma'$ runs through all of  $\Hom_{K_1}(L_1,L_1^{alg})$.
The second equality follows from Lemma \ref{M_1=M_2}.(2) because $[K_2(g(\pi_1)):K_2]$ is equal to $[L_1:K_1]$ and $g(\pi_1)$ is a uniformizer of $K_2(g(\pi_1))$ by Fact \ref{fact:valuationring_homo_isometry}.

\smallskip

(2)
For any $x$ in $R_1$, we can write
$x= \sum_{r=0}^{\infty}\lambda_r \pi_1^r $ where $\lambda_r$ is in $S_1$ for each $r$.
Then $$
\phi(x+\fm_1^{n_1}) =\phi\left( \left(\sum_{r=0}^{\infty}\lambda_r \pi_1^r \right)+\fm_1^{n_1}\right) \\
=\left(\sum_{r=0}^{\infty}\tau_r \beta^r \right)+\fm_2^{n_2}
$$
where $\tau_r$ is a representative of $\phi(\lambda_r +\fm_1^{n_1})$ contained in $S_2$ which is guaranteed by Lemma \ref{restriction map}.(1).
 In particular $\sum_{r=0}^{\infty}\tau_r \beta^r $ is a representative of $\phi(x+\fm_1^{n_1}) $, say $\beta_x$.
By  Lemma \ref{restriction map}.(1) again, we have $g(\lambda_r)=\tau_r$, and hence,
$$
g(x) =g\left(\sum_{r=0}^{\infty}\lambda_r \pi_1^r\right)=\sum_{r=0}^{\infty}\tau_r g(\pi_1)^r.
$$
We obtain
\begin{align*}
\widetilde{\nu_2}(g(x)- \beta_x)&=\widetilde{\nu_2}\left(\sum_{r=0}^{\infty}\tau_r g(\pi_1)^r-\sum_{r=0}^{\infty}\tau_r \beta^r\right)\\
&=\widetilde{\nu_2}\left( \Big(g(\pi_1)-\beta \Big)\sum_{r=1}^{\infty} \tau_r\left(\sum_{j=0}^{r-1} g(\pi_1)^{r-1-j}\beta^j\right)  \right)\\
& >M(R_1)
\end{align*}
because
\begin{align*} \widetilde{\nu_2}\left(g(\pi_1)-\beta\right)& >\max_{\sigma}\left\{\widetilde{\nu_2}\left(\sigma\left(g(\pi_1)\right) -\beta  \right):\sigma\left(g(\pi_1)\right)\neq g(\pi_1)\right\}\\
&=M(R_1).
\end{align*}
So $g$ is a $(n_1, n_2)$-lifting of $\phi$.
\end{proof}

The following theorem shows that there is a unique lifting if we enlarge the lengths of residue rings.
\begin{theorem}\label{main_theorem_construction}
Suppose  $n_2> M(R_1)\nu_1(p)\nu_2(p) $ and $\Hom(R_{1,(n_1)},R_{2,(n_2)})$ is not empty.
Then there exists a unique $(n_1,n_2)$-lifting $\L_{n_1, n_2}:\Hom(R_{1,(n_1)},R_{2,(n_2)})\longrightarrow\Hom(R_1,R_2)$.
Also, $\L_{n_1, n_2}(\phi)$ is an isomorphism when $\phi:R_{1,(n_1)}\longrightarrow R_{2,(n_2)}$ is an isomorphism.
\end{theorem}
\begin{proof}
Let $\phi$ be a homomorphism from $R_{1,(n_1)}$ to $R_{2,(n_2)}$.
By Lemma \ref{restriction map}.(2), let
$$
\phi_{res} :\frac{W(k_1)+\fm_1^{n_1}}{\fm_1^{n_1}} \longrightarrow \frac{W(k_2)+\fm_2^{n_2}} {\fm_2^{n_2}}
$$
be the restrition map of $\phi$.
For an element $a=\sum_{r=0}^{\infty} \lambda_r p^r$ in $W(k_1)$, we define $g_{res}:W(k_1)\longrightarrow W(k_2)$ by the rule
$$
g_{res}:W(k_1)\longrightarrow W(k_2),\ a\mapsto g_{res}(a) =\sum_{r=0}^{\infty}\tau_r p^r
$$
 where $\tau_r$ is a unique representative of $\phi_{res} (\lambda_r + \fm_1^{n_1})$ which is contained in $S_2$, the set of Teichm\"{u}ller representatives of $R_2$.
Then, by the proof of Fact \ref{Witt_classic}.(2) (c.f. the proof of \cite[Proposition 10]{S}), $g_{res}$ is a homomorphism and $g_{res}$ induces $\phi_{res}$.
By Fact \ref{witt_embedding}, $L_1=K_1(\alpha)$ is totally ramified of degree $\nu_1(p)$ over $K_1$, that is, $[L_1:K_1]=\nu_1(p)$, where $\alpha=\pi_1$ is a uniformizer of $R_1$.
Let $f$ be the monic irreducible polynomial of $\alpha$ over $K_1$.
The ring homomorphism $g_{res}$ induces a field homomorphism from $K_1$ into $K_2$.
We still denote the fraction field homomorphism by $g_{res}$ if there is no confusion. Then $g_{res}: K_1 \longrightarrow K_2$ is an isometry by Fact \ref{fact:valuationring_homo_isometry}.
Let $\widetilde{g_{res}}: K_1^{alg}\longrightarrow K_2^{alg}$ be an  extended field homomorphism of $g_{res}$, which is also an isometry by Fact \ref{isometry}. Write
\begin{align*}
f=&x^{\nu_1(p)} + \cdots+ a_1 x + a_0 \\
=&(x-\alpha_1)\cdots(x-\alpha_{\nu_1(p)})
 \end{align*}
 where $\alpha=\alpha_1$, and put
 \begin{align*}
 g_{res}(f)=& x^{\nu_1(p)} + \cdots + g_{res}(a_1) x + g_{res}(a_0)\\
=&\left(x-\widetilde{g_{res}}(\alpha_1)\right)
\cdots \left(x-\widetilde{g_{res}}(\alpha_{\nu_1(p)})\right).
\end{align*}
We have that $[K_2(\wi{g_{res}}(\alpha) ):K_2]\le [K_1(\alpha) :K_1] = \nu_1(p)$ and that $\wi{\nu_2}(\wi{g_{res}}(\alpha)) = \wi{\nu_1}(\alpha)= 1/ \nu_1(p)  $ because $\wi{g_{res}}$ is an isometry.
Therefore $g_{res}(f)$ is the monic irreducible polynomial of $\widetilde{g_{res}} (\alpha)$ over $K_2$.
Let $\beta$ be any representative of $\phi(\alpha+ \fm_1^{n_1})$. Since $g_{res}$ induces $\phi_{res}$,
we can write
\begin{align*}
0+\fm_2^{n_2}&=\phi(f(\alpha)+ \fm_1^{n_1} )\\
&=\phi(\alpha+\fm_1^{n_1})^{\nu_1(p)} + \cdots+ \phi(a_1 + \fm_1^{n_1})\phi(\alpha+\fm_1^{n_1}) + \phi(a_0 + \fm_1^{n_1})\\
&=g_{res}(f)(\beta) +\fm_2^{n_2}.
\end{align*}
This shows that
$g_{res}(f)(\beta)$ is in $\fm_2^{n_2}$
and
$$
\nu_2\left(g_{res}(f)(\beta)\right) \ge n_2> M(R_1)\nu_1(p)\nu_2(p).
$$
We claim that there exists an index $i_0$ satisfying $\widetilde{\nu_2}(\beta -\widetilde{g_{res}}(\alpha_{i_0}) )> M(R_1) $.
If
$\widetilde{\nu_2}(\beta -\widetilde{g_{res}}(\alpha_i) ) \le M(R_1)$ for all $i$, then
$$
\widetilde{\nu_2}\left(g_{res}(f)(\beta)\right)
=\widetilde{\nu_2}\left( \prod_{i}\left(  \beta -\widetilde{g_{res}}(\alpha_i)  \right)  \right)\le M(R_1)\nu_1(p).
$$
This shows
$$
n_2\le \nu_2\left( g_{res}(f)(\beta) \right)
= \nu_2(p)\widetilde{\nu_2}\left(g_{res}(f)(\beta)\right)
\le M(R_1)\nu_1(p)\nu_2(p),
$$
which is impossible. Thus there is an index $i_0$ satisfying
$$
\widetilde{\nu_2}\left(\beta -\widetilde{g_{res}}(\alpha_{i_0}) \right)> M(R_1)=\max\left\{\widetilde{\nu_2}\left(\widetilde{g_{res}} (\alpha_1)- \widetilde{g_{res}} (\alpha_j)\right): j=2,...,\nu_1(p) \right\}
$$
where the equality follows from the fact that $\widetilde{g_{res}}$ is an isometry.
Hence, by Fact \ref{Krasner},  $K_2(\widetilde{g_{res}}(\alpha_{i_0})) \subset K_2(\beta) \subset L_2$.
We define an extended homomorphism $g:L_1 \longrightarrow L_2$ of $g_{res}:K_1\longrightarrow K_2$ by the rule
$\pi_1\mapsto g(\pi_1)=\widetilde{g_{res}}(\alpha_{i_0}) $. Then,
$g$ induces the restricted homomorphism from $R_1$ to $R_2$ which is still denoted by $g$.
Also, $g$ is a $(n_1,n_2)$-lifting of $\phi$ because $g_{res}$ induces $\phi_{res}$ and
$$
M(R_1)=\max_{\sigma}\left\{\widetilde{\nu_2}\left( \sigma\left(g(\pi_1)\right) -\beta  \right) :\sigma\left(g(\pi_1)\right)\neq g(\pi_1) \right\}
$$ by Lemma \ref{proposition_lifting homomorphism}.

Suppose that  $g_1:R_1\longrightarrow R_2$ is an $(n_1,n_2)$-lifting of $\phi$ other than $g$. We note that the restriction $g|_{S_1}$ of $g$ to $S_1$ is equal to   $g_1|_{S_1}$ by Fact \ref{Witt_classic}.
From  Remark \ref{rem:unique_expresseion_Teichmuller} and  $g|_{S_1}=g_1|_{S_1}$, it follows that $g_1|_{W(k_1)}=g|_{W(k_1)}$. Since $R_1=W(k_1)[\pi_1]$, $g=g_1$ if $g(\pi_1)=g_1(\pi_1)$. So, $g(\pi_1)\neq g_1(\pi_1)$, and by Proposition \ref{proposition_lifting homomorphism},
$$
\widetilde{\nu_2} \left(g_1 (\pi_1) -\beta\right) >\max_{\sigma}\left\{\widetilde{\nu_2}\left(\sigma\left(g_1(\pi_1)\right) -\beta  \right): \sigma \left( g_1\left(\pi_1\right) \right)\neq g_1(\pi_1) \right\}.
$$
Since $g_1|_{W(k_1)} = g|_{W(k_1)}$, $g(\pi_1)$ and $g_1(\pi_1)$ have the same minimal polynomial over $W(k_2)$ and
$$
\left\{\sigma\left(g_1(\pi_1)\right):\sigma\in \Hom_{K_2}(L_2, L_2 ^{alg})\right\}=\left\{\sigma\left(g(\pi_1)\right) :\sigma\in \Hom_{K_2}(L_2, L_2 ^{alg})\right\}.
$$
In particular
$g_1(\pi_1)=\sigma(g(\pi_1))$ for some $\sigma\in \Hom_{K_2}(L_2, L_2 ^{alg})$.
Since $g_1(\pi_1) \neq g(\pi_1)$, we have the inequalities $\wi{\nu_2}(g_1(\pi_1) -\beta) > \wi{\nu_2}(g(\pi_1) -\beta)$ and $\wi{\nu_2}(g_1(\pi_1) -\beta) < \wi{\nu_2}(g(\pi_1) -\beta)$ simultaneously by the first bullet point of Proposition \ref{proposition_lifting homomorphism}. This gives a contradiction, and hence, we obtain the uniqueness of the lifting.

When $\phi$ is an isomorphism,  so are $\phi_{res}$ and $g_{res}$. We obtain $[L_2:K_2]  =[L_1:K_1]$ from the assumption that $n_2>M(R_1)\nu_1(p)\nu_2(p)$, and hence, $\L_{n_1, n_2} (\phi)$ is also an isomorphism.
\end{proof}

\noindent
We note that the proof of Theorem \ref{main_theorem_construction} works for any representative $\beta$ of $\phi(\pi_1+\fm_1^{n_1})$.

\begin{example}\label{ex:example_tame optimal}
Let $R_1=\BZ_3[\sqrt{3}]$ and $R_2=\BZ_3[\sqrt{-3}]$.
There is no homomorphism between $R_1 $ and $R_2$ by Kummer theory. But there is an isomorphism
$$
\phi:R_{1,(2)}= \frac{\BZ_3[\sqrt{3}]}{3\BZ_3[\sqrt{3}]}
\longrightarrow
R_{2,(2)}=\frac{\BZ_3[\sqrt{-3}]}{3\BZ_3[\sqrt{-3}]}
$$
given by the rule
$a+b\sqrt {3} \mapsto a+b\sqrt{-3}$. Since $\nu_1(3)=\nu_2(3)=2$ and $M(R_1)=\wi{\nu_1}(\sqrt{3}-(-\sqrt{3}))=1/2$, we obtain $M(R_1)\nu_1(3)\nu_2(3)=2$. Hence
the lower bound for $n_2$ in \emph{Theorem \ref{main_theorem_construction}} is the best possible in this case. This phenomenon will be generalized in \emph{Proposition \ref{tame optimal pro}}.
\end{example}

We give a generalized version of Fact \ref{Witt_classic}.(1) for the ramified case. We first give a useful upper bound for  $M(R)$.
\begin{lemma}\label{lem:upperbound_M(R)}
Let $(R,\nu,k)$ be a valuation ring and let $\pi$ be a uniformizer of $R$. Let $L$ and $K$ be fraction fields of $R$ and $W(k)$ respectively. Then,
$$M(R)\le \frac{1+\nu(\nu(p))}{\nu(p)}.$$
\end{lemma}
\begin{proof}
Let $f$ be the monic irreducible polynomical of $\pi$ over $K$,   which is of degree  $e:=\nu(p)$.
 Let $\pi_1(:=\pi),\ldots, \pi_e$ be the distinct zeros of $f$.
   We have
$\widetilde{\nu}(\pi)= 1/e$ and hence  $\widetilde{\nu}
(\pi_i-\pi_j)\ge 1/e$ for all $i$ and $j$.
Furthermore, by definition of $M(R)$, we have that for some $2\le i_0\le e$,
\begin{itemize}
	\item $M(R)\ge \wi{\nu}(\pi)=\frac{1}{e}$; and
	\item $M(R)=\wi{\nu}\big(\pi_1-\pi_{i_0}\big)$.
\end{itemize}

Consider the differentiation
$$
f'= \sum_{i=1}^{e}{\frac{f}{(x-\pi_i)}}.
 $$
There are two cases.
\begin{itemize}
\item
Tame case:
  Suppose $L/K$ is tamely ramified. Hence, $\nu(\nu(p))=\nu(e)=0$.
It follows from Fact \ref{different bound} that
$$
\frac{e-1}{e} = \widetilde{\nu}\big(f' (\pi_1)\big)
=\widetilde{\nu}\left(\prod_{j\neq 1}(\pi_1-\pi_j)\right)
=\sum_{j \neq 1} \widetilde{\nu}(\pi_1-\pi_j).
$$
Since $\widetilde{\nu}
(\pi_i-\pi_j)\ge 1/e$
,
$\widetilde{\nu}
(\pi_1-\pi_j)
=1/e=M(R)
$ for $j\neq1$.
Hence, we have that $$
M(R)=\frac{1}{e}=\frac{1+\nu(e)}{e}.$$

\item
Wild case:
  Suppose $L/K$ is wildly ramified.
Noting that  $\widetilde{\nu}
(\pi_i-\pi_j)\ge 1/e$, we have that
\begin{align*}
M(R)&\le \wi{\nu}(\pi_1-\pi_{i_0})+\sum_{2\le i\neq i_0\le e}\left(\wi{\nu}(\pi_1-\pi_{i})-\frac{1}{e}\right)\\
&= \wi{\nu}\left(\prod_{i\neq 1}(\pi_1-\pi_i)\right)-\frac{(e-2)}{e}\ =\ \wi{\nu}(f'(\pi))-\frac{e-2}{e}\\
&\le \frac{e-1+\nu(e)}{e} -\frac{e-2} {e} \ =\ \frac{1+\nu(e)}{e}
\end{align*}
by Fact \ref{different bound} again.
\end{itemize}
Therefore we get the desired result.
\end{proof}

\begin{theorem}\label{main_theorem_existence}
Let $\overline{R}$ be a principal Artinian local ring of length $n$ with  perfect residue field $k$ of characteristic $p$ and  maximal ideal $\overline{\mathfrak{m}}$, that is, $\overline{\mathfrak{m}} ^n=0$ and $\overline{\mathfrak{m}} ^{n-1}\neq0$.
Suppose that $\overline{R}$ has no finite subfield as a subring. For any positive integer $a$, if $a$ generates a nonzero ideal $\overline{\mathfrak{m}}^{k}$, we denote $k$ by $\nu(a)$.
Suppose
 $$
 \nu(p)\overline{R}\neq0\ \mbox{and }n>\nu(p)+\nu(p) \nu\big(\nu(p) \big) .
 $$ Then there exists a complete discrete valuation ring of characteristic $0$ which has $\overline{R}$ as its $n$-th residue ring.
 Also such a ring is unique up to isomorphism.
\end{theorem}

\begin{proof}
Any principal Artinian local ring is a homomorphic image of a discrete valuation ring. This can be proved by Cohen structure theorem for complete local rings (c.f. \cite{H}) or, more directly, by the property of CPU-rings (c.f. \cite{HT}). Since the completion of a discrete valuation ring $R$ has the same $n$-th residue ring as that of $R$, we may assume that there are complete discrete valuation rings $R_1 $ and $R_2$  such that $R_{1,(n)}$ and $R_{2,(n)}$ are isomorphic to $\overline{R}$. We note that $R_i$ is of characteristic $0$ for $i=1,2$ because $\overline{R}$ has no finite subfield as a subring.
Let $L_i$ and $K_i$ be the fraction fields of $R_i$ and $W(k_i)$ for $i=1,2$ respectively.
Then by Fact \ref{witt_embedding}, $L_1 =K_1(\alpha)$ where $\alpha=\pi_1$ is a uniformizer of $R_1$.
By Lemma \ref{lem:upperbound_M(R)}, we have that $$M(R_1)\nu_1(p)\nu_2(p)\le \nu_2(p)(1+\nu_1(\nu_1(p)))=\nu(p)(1+\nu(\nu(p)).$$
Note that $\nu(\nu(p))$ and $\nu(p)$ are well-defined since $\nu(p)\overline{R}\ne 0$ and $\overline{R}$ has no finite subfield. The desired result follows from Theorem \ref{main_theorem_construction}.
\end{proof}
\noindent
Note that the notation $\nu(p)$ in Theorem \ref{main_theorem_existence} is compatible with the previously defined valuation.
Suppose that a discrete valuation ring $R$ with valuation $\nu$
and maximal ideal $\mathfrak{m}$
has $\overline{R}$ as its residue ring. Then $\nu(p)$ is equal to a  power of the maximal ideal generated by $p$, that is, $pR= \mathfrak{m}^{\nu(p)}$ as we noted in the proof of Theorem \ref{main_theorem_existence}.


\section{Functoriality}\label{Functoriality}
The main purpose of this section is to give a generalized version of Fact \ref{Witt_modern} for the ramified case.
For a prime number $p$ and a positive integer $e$, let $\CC_{p,e}$ be a category consisting of the following data:
\begin{itemize}
	\item $\ob(\CC_{p,e})$ is the family of complete discrete valuation rings of mixed characteristic having perfect residue fields of characteristic $p$ and the ramification index $e$; and
	\item $\mor_{\CC_{p,e}}(R_1,R_2):=\Hom(R_1,R_2)$ for $R_1$ and $R_2$ in $\ob(\CC_{p,e})$.
\end{itemize}
Let $\CR_{p,e}^{n}$ be a category consisting of the following data:
\begin{itemize}
	\item For $n\le e$, $\ob(\CR_{p,e}^{n})$ is the family of principal Artinian local rings $\overline{R}$ of length $n$ with perfect residue fields of characteristic $p$, and for $n>e$, $\ob(\CR_{p,e}^{n})$ is the family of principal Artinian local rings $\overline{R}$ of length $n$ with perfect residue fields of characteristic $p$ such that $p\in \overline{\fm}^{e}\setminus\overline{\fm}^{e+1}$ where $\overline{\fm}$ is the maximal ideal of $\ov{R}$; and
	\item $\mor_{\CR_{p,e}^{n}}(\overline{R_1},\overline{R_2}) :=\Hom(\overline{R_1},\overline{R_2})$ for $\overline{R_1}$ and $\ov{R_2}$ in $\ob(\CR_{p,e,}^{n})$,
\end{itemize}
Note that for $e_1,e_2\ge 1$ and for $n\le e_1,e_2$, two categories $\CR_{p,e_1}^n$, $\CR_{p,e_2}^n$ are the same. For each $m>n$, let $\Pr_n:\ \CC_{p,e}\rightarrow \CR_{p,e}^n$ and $\Pr^m_n:\ \CR_{p,e}^m\rightarrow \CR_{p,e}^n$ be the canonical functors respectively.

\begin{definition}\label{n-liftable}
Fix a prime number $p$ and a positive integer $e$.
\begin{enumerate}
	\item
We say that the category $\CC_{p,e}$ is \emph {$n$-liftable}  if there is a  functor $\L:\CR_{p,e}^n \longrightarrow \CC_{p,e}$ which satisfies the following:
 \begin{itemize}
 \item
 $(\Pr_n\circ \L)(\ov{R})\cong \ov{R}$ for each $\overline{R}$ in $\ob(\CR_{p,e})$.
  \item
 $\Pr_1 \circ\L$ is equivalent to $\Pr^{n}_1$.
 \item
 $\L\circ \Pr_n$ is equivalent to  $\id_{\CC_{p,e}}$, the identity functor.

 \end{itemize}
 We say that $\L$ is a \emph{$n$-th lifting functor} of $\CC_{p,e}$.
	\item The \emph{lifting number} for $\CC_{p,e}$ is the smallest positive integer $n$ such that $\CC_{p,e}$ is $n$-liftable. If there is no such $n$, we define the lifting number for  $\CC_{p,e}$ to be $\infty$.
\end{enumerate}
\end{definition}
\noindent
We note that the condition  $(\Pr_n\circ \L)(\ov{R})\cong \ov{R}$ in  the first bullet point in Definition \ref{n-liftable}.(1) is weaker than the condition that $\Pr_n\circ \L$ is equivalent to $\id_{\CR_{p,e}^n}$.
By Example \ref{example_tame optimal}, $\Pr_n\circ \L$ is not equivalent to $\id_{\CR_{p,e}^n}$ in general.

\begin{remark}\label{basic_on_n-liftable}
\begin{enumerate}
  \item
Suppose that  there is a $n$-th lifting functor $\L :\CR_{p,e}^n\rightarrow \CC_{p,e}$.
For any $\ov{R}$ in $\ob(\CR_{p,e})$, $\L (\ov{R})$ is the unique (up to isomorphism) object in $\ob(\CC_{p,e})$ which has $\ov{R}$ as its $n$-th residue ring.
Indeed, suppose that $R$ in $\ob(\CC_{p,e})$ has  $\ov{R}$ as its $n$-th residue ring. Since $\L\circ \Pr_n$ is equivalent to the identity functor $\id_{\CC_{p,e}}$, $R=\id_{\CC_{p,e}}(R)$ is isomorphic to $(\L\circ \Pr_n)(R)=\L (\ov{R})$.

	\item
The lifting number for $\CC_p$ is $1$ by \textup{Fact} \ref{Witt_modern}. We will see that the lifting number for $\CC_{p,e}$ is always larger than $e$ whenever $e>1$ in \emph{Corollary \ref{lifting number >e}}.

\item
For $n\ge e$, a functor $\L_{n+1}:=\L_n\circ \Pr^{n+1}_n$  is a $(n+1)$-th lifting functor of $\CC_{p,e}$ for any $n$-th lifting functor $\L_n :\CR_{p,e}^n\rightarrow \CC_{p,e}$.
The proof is as follows: For $\ov{R}$ in $\ob(\CR_{p,e}^{n+1})$, there exists a ring $R$ in $\ob(\CC_{p,e})$ which satisfies $\Pr_{n+1}(R)=\ov{R}$ as noted in the proof of \emph{Theorem \ref{main_theorem_existence}}.
 Since there is a unique object in $\ob(\CC_{p,e})$ which has  $\Pr_n(R)$ as its $n$-th residue ring by \emph{Remark \ref{basic_on_n-liftable}.(1)}, we have that
$$({\Pr}_{n+1}\circ\L_{n+1})\left(\ov{R}\right) ={\Pr}_{n+1}\circ(\L_n\circ {\Pr}^{n+1}_{n})\left(\ov{R}\right)={\Pr}_{n+1}(R)=\ov{R}.$$
Also,
${\Pr}_1 \circ {\L}_{n+1}= ({\Pr}_1 \circ {\L}_n )\circ {\Pr}_n^{n+1}$ is equivalent to ${\Pr}_1^n \circ {\Pr}_n^{n+1}= {\Pr}_1^{n+1}$ and
$$
{\L}_{n+1}\circ{\Pr}_{n+1} =({\L}_{n}\circ{\Pr}_n^{n+1})\circ{\Pr}_{n+1}
  ={\L}_{n}\circ{\Pr}_n
  $$
   is equivalent to $\id_{\CC_{p,e}}$.

\end{enumerate}
\end{remark}


\begin{proposition}\label{proposition lifting functor}
For $1\le i\le3$, let $(R_i,\fm_i,\nu_i)$ be a complete discrete valuation ring of mixed characteristic $(0,p)$ with a  perfect residue field and let $\pi_i$ be a uniformizer of $R_i$.
 For $\phi^{1,2}:R_{1,(n_1)}\longrightarrow R_{2,(n_{2})}$ and $\phi^{2,3}:R_{2,(n_2)}\longrightarrow R_{3,(n_{3})}$, suppose that there are liftings $g^{1,2}:R_1\longrightarrow R_{2}$ and $g^{2,3}:R_2\longrightarrow R_{3}$  of $\phi^{1,2}$ and $\phi^{2,3}$ respectively.

 If $\nu_1(p)=\nu_2(p)$, then $g=g^{2,3}\circ g^{1,2}$ is a  lifting of $\phi^{2,3}\circ \phi^{1,2}$. Moreover $g$ is the unique lifting of $\phi^{2,3}\circ \phi^{1,2}$ when $n_3>M(R_2)\nu_2(p)\nu_3(p)$.
\end{proposition}

\begin{proof}
By Fact \ref{fact:valuationring_homo_isometry}, the liftings $g^{1,2}$ and $g^{2,3}$ are isometries.
Also, since both $\widetilde{\nu_2}$ and $\widetilde{\nu_3}$ are normalized, we have $\widetilde{\nu_3}(g^{2,3}(x))=\widetilde{\nu_2}(x)$ for any $x\in R_2$. By Lemma \ref{M_1=M_2}, $M(R_1)=M(R_2)$, say $M$.
Since $g^{1,2}$ is a lifting of $\phi^{1,2}$, there is a representative $\beta_1$  of $\phi^{1,2}(\pi_1+\fm_1^{n_1})$
such that $\wi{\nu_2}(g^{1,2} (\pi_1)-\beta_1)>M$.
We note that $\beta_1$ is a uniformizer of $R_2$. Since
$g^{2,3}$ is a lifting of $\phi^{2,3}$, there is a representative $\beta_2$  of $$(\phi^{2,3} \circ \phi^{1,2} ) (\pi_1+ \fm_1^{n_1})=\phi^{2,3}(\beta_1 +\fm_2^{n_2} )$$
such that $\wi{\nu_3}(g^{2,3} (\beta_1)-\beta_2)>M$.

 If we write $g^{1,2} (\pi_1) =\beta_1 +x_{M}$ where $\widetilde{\nu_2}(x_{M}) >M$, then $$g(\pi_1)=g^{2,3} (g^{1,2}(\pi_1))
=g^{2,3}(\beta_1 + x_M).$$
Since $\widetilde{\nu_3}(g^{2,3}(\beta_1) -\beta_2 )>M$ and $\widetilde{\nu_3}(g^{2,3}(x_M))=\widetilde{\nu_2}(x_M )>M$,
$$
\widetilde{\nu_3} \big(g(\pi_1)-\beta_2\big)
=\widetilde{\nu_3} \big(g^{2,3}(\beta_1)-\beta_2+ g^{2,3}(x_M)\big)
>M.
$$
The equality
$(\phi^{2,3}\circ \phi^{1,2})_{red,1} \circ \pr_{1,1}=\pr_{3,1} \circ g$ follows directly from $g=g^{2,3} \circ g^{1,2}$. By  Proposition \ref{proposition_lifting homomorphism}, $g$ is a lifting of $\phi^{2,3}\circ \phi^{1,2}$.

When $n_3>M(R_2)\nu_2(p)\nu_3(p)=M(R_1)\nu_1(p)\nu_3(p)$,
$g$ is the unique lifting of $\phi^{2,3} \circ \phi^{1,2}$ by Theorem \ref{main_theorem_construction}.
\end{proof}

\begin{theorem}\label{main theorem lifting functor}

The lifting number for $\CC_{p,e}$ is finite.
More precisely, $\CC_{p,e}$ is $(e+e\nu(e)+1)$-liftable.
Here $\nu(e)$ denotes the exponent $n$ such that $e$ generates an ideal $\mathfrak{m}^{n}$ of  $R$ in $ \ob(\CC_{p,e})$ where $\mathfrak{m}$ denotes the maximal ideal of $R$.
The value $\nu(e)$ depends only on the prime number $p$ and the ramification index $e$, in particular $\nu(e)$ is independent of the choice of $R$ in $ \ob(\CC_{p,e})$.

\end{theorem}

\begin{proof}
Suppose $n$ is bigger than $e+e\nu(e)$. For any $\ov{R}, \ov{R_1}$ and $\ov{R_2}$ in $\ob(\CR_{p, e}^n)$, by Theorem \ref{main_theorem_existence}, we define $\L_n (\ov{R})$ to be a unique ring $R$ in $\ob(\CC_{p,e})$  which satisfies $\Pr_n(R)=\ov{R}$. By Lemma \ref{lem:upperbound_M(R)}, $e+e\nu(e)\ge M(R)e^2$.
By Theorem \ref{main_theorem_construction}, for any $\phi:\ov{R_1} \longrightarrow \ov{R_2}$, there exists a unique $n$-th lifting map $\L(\phi):\L(\ov{R_1})\longrightarrow \L(\ov{R_2})$, and hence we obtain a lifting functor $\L_n:\CR_{p,e}^n\longrightarrow\CC_{p,e}$ by Proposition \ref{proposition lifting functor}.
\end{proof}

\noindent Example \ref{ex:example_tame optimal} can be generalized as follows.
\begin{proposition}\label{tame optimal pro}
Let $R_1/W(k)$ and $R_2/W(k)$ be totally  ramified extensions of degree $e$. Then $R_{1,(e)}$ is isomorphic to $R_{2,(e)}$ as $W(k)$-algebras.
\end{proposition}

\begin{proof}
Let $\pi_i$  be a uniformizer of $R_i$ and let $\nu_i$ be the valuation corresponding to $R_i$ for $i=1,2$. By the theory of totally ramified extensions (see Chapter 2 of \cite{L} for example), the monic irreducible polynomial $f_i$ of $\pi_i$ over $W(k)$ is an Eisenstein polynomial for $i=1, 2$. If we write
$f_i=x^e + a_{i,e-1}x^{e-1}+ \cdots +a_{i,1}x+a_{i,0}$, then $\nu_i(p)=\nu_i(a_{i,0})=e$ and $\nu_i(a_{i,j})\ge e$ for $i=1,2$ and $j=1,2, \dots ,e-1$. This shows

\begin{align*}
R_{i,(e)}
&=\frac{W(k)[\pi_i]}{(\pi_i)^e}\  \cong\  \frac{W(k)[x]}{(p,f_i)}\\
&=\frac{k[x]}{(x^e + \cdots +a_{i,1}x+a_{i,0})}\\
&=\frac{k[x]}{(x^e)},
\end{align*}

and hence, $R_{1,(e)}$ is isomorphic to $R_{2,(e)}$ as $W(k)$-algebras.
\end{proof}

For the tame case, we can calculate the lifting number. We denote a primitive $n$-th root of unity by $\zeta_n$.

\begin{lemma}\label{tame distinct}
Let $k$ be a perfect field of characteristic $p$ and let $K$ be the fraction field of $W(k)$. Let $e$ be a positive integer prime to $p$. Suppose that there is a prime divisor $l$ of $e$ such that $\zeta_{l^n}$ is in $k^{\times}$ and $\zeta_{l^{n+1}}$ is not in $k^{\times}$ for some $n>0$. Then there are two  totally ramified extensions $L_1$ and $L_2$ of degree $e$ over $K$ which are not isomorphic over $\BQ$.
\end{lemma}

\begin{proof}
We have $\zeta_{l^n}$ is in  $W(k)^{\times}$ by Hensel's lemma, and $\zeta_{l^{n+1}}$ is not in $W(k)^{\times}$. Then $L_1 = K(\sqrt[e]{p})$ and $L_2=K(\sqrt[e]{p\zeta_{l^n}})$ are totally  ramified extensions of degree $e$ over $K$. Suppose that there is an  isomorphism $\sigma:L_2\longrightarrow L_1$. Since Galois conjugates of $\sqrt[e]{p}$ and $\zeta_{el^{n}}$ over $\BQ$ are of the form $\sqrt[e]{p}\zeta_{e}^i$ and $\zeta_{el^{n}}^j$ respectively for some $i$ and $j$ with $(j,e)=1$,
 $$
 \sigma \left(\sqrt[e]{p\zeta_{l^n}}\right)=\sigma \left(\sqrt[e]{p}\zeta_{el^{n}}\right)=\sqrt[e]{p}\zeta_{el^{n}}^k
 $$
 for some $k$ prime to $l$. In particular, $L_1$ contains both $\sqrt[e]{p}$ and $\sqrt[e]{p}\zeta_{el^{n}}^k$,
 and hence, $\zeta_{l^{n+1}}$ is in $L_1$. This is a contradiction because $L_1/K$ is totally ramified.
\end{proof}

\begin{corollary}\label{tame optimal cor}
Suppose that $p$ does not divide $e$ and $e>1$. Then $e+1$ is the lifting number for $\CC_{p,e}$.
\end{corollary}

\begin{proof}
Since $\nu(p)=0$,
$e+e\nu(e)+1=e+1$. By Theorem \ref{main theorem lifting functor}, $\CC_{p,e}$ is $(e+1)$-liftable.
Let $\BF_p$ be the prime field of $p$ elements. Let $K$ be the fraction field of the Witt ring $W(k)$ of $k=\BF_p(\zeta_e)$.
 By Lemma \ref{tame distinct}, there are two totally ramified extensions $L_1$ and $L_2$ of degree $e$ over $K$ such that there is no isomorphism between $L_1$ and $L_2$.
 If $\CC_{p,e}$ is $e$-liftable, $L_1$ and $L_2$ are isomorphic over $K$ by  Proposition \ref{tame optimal pro} and it is a contradiction.
\end{proof}
\noindent

\begin{remark}

\emph{Proposition \ref{tame optimal pro}} and \emph{Corollary \ref{tame optimal cor}} show the difference between the unramified case and the tamely ramified case.
We can regard the unramified valued fields of mixed characteristic as the  tamely ramified valued fields having the ramification index $e=1$. If we apply  \emph{Corollary \ref{tame optimal cor}} to $\CC_p$, the lifting number for $\CC_p$ should be $1+1=2$.
However the argument in the proof of \emph{Corollary \ref{tame optimal cor}} does not work for $\CC_p$. For an  unramified  complete discrete valued field $K$,  there is a unique totally ramified extension  of degree $1$ over $K$, that is, $K$ itself.
 Hence the fact that the lifting number for $\CC_p$ is $1$ does not contradict  \emph{Corollary \ref{tame optimal cor}}.

\end{remark}
\bigskip

For the wild case, we have the following example.
Let $R_1=\BZ_2[\sqrt{2}]$ and $R_2=\BZ_2[\sqrt{10}]$. There is no homomorphism between $R_1$ and $R_2$ by Kummer theory. But there is an isomorphism between $R_{1,(6)}$ and $R_{2,(6)}$ because
\begin{align*}
R_{1,(6)}&=\frac{\BZ_2[\sqrt{2}]}{({\sqrt2}^6)}\ \cong\ \frac{\BZ_2 [x]}{(x^2-2, 8)}\\
&=\frac{\BZ_2 [x]}{(x^2-10, 8)}\ \cong\ \frac{\BZ_2[\sqrt{10}]}{({\sqrt2}^6)}\ =\ R_{2,(6)}.
\end{align*}
Note that the last equality holds because $(\sqrt{10})^6\BZ_2[\sqrt{10}]=(\sqrt{2})^6\BZ_2[\sqrt{10}]$. This shows that the lifting number for $\CC_{2, 2}$ is $2+2\nu(2)+1=7$ by Theorem \ref{main theorem lifting functor}. In general, we have a lower bound $e+1$ of the lifting number for the wild case. To prove this, we need the following lemma.

\begin{lemma}\label{wild distinct}
Let $k$ be a perfect field of characteristic $p$ and let $K$ be the fraction field of the Witt ring $W(k)$ of $k$.
 Let $e$ be a positive integer divisible by $p$.
Then there are two  totally ramified extensions $L_1$ and $L_2$ of degree $e$ over $K$ which are not isomorphic over $\BQ$.
\end{lemma}

\begin{proof}

We write $e=sp^r$ for some positive integers $s$ and $r$ where $s$ is prime to $p$. Let $\BQ_{\infty}/\BQ$ be the cyclotomic $\BZ_p$-extension, in particular $\mathrm{Gal}(\BQ_{\infty}/\BQ)\cong \BZ_p$.
Let $M_r$ be a unique subfield of $\BQ_{\infty}$ such that $[M_r:\BQ]=p^r$.
 By the theory of cyclotomic fields (c.f. \cite[Chapter 1]{N}), the Galois extension $M_r/\BQ$ is  totally  ramified  at the place above $p$. Let $\alpha$ be a uniformizer of $M_r$ corresponding to the place above $p$. Since $M_r/\BQ$ is a Galois extension, $M_r=\BQ(\alpha)=\BQ(\sigma(\alpha))$ for any embedding $\sigma$. We fix an embedding $\BQ^{alg} \subset K^{alg}$.

Let $L_1=K(p^{1/e})=K(p^{1/s},p^{1/p^r} )$ and $L_2=K(p^{1/s},\alpha )$. Then $L_1$ and $L_2$ are totally  ramified extensions of degree $e$ over $K$.
If there is an isomorphism $\sigma:L_2 \longrightarrow L_1$, $L_1$ contains both $\sigma(\alpha)$ and $p^{1/p^r}$. Since $\BQ(\alpha)=\BQ(\sigma(\alpha))$, $K(\sigma(\alpha))=K(\alpha)$ is contained in $L_1$. We note that $[K(p^{1/p^r},\alpha):K(p^{1/p^r})]$ divides $[K(\alpha):K]=p^r$ because $K(\alpha)/K$ is a Galois extension.
 Since
$$
s=\left[L_1:K\left(p^{1/p^r}\right)\right] =\left[L_1:K\left(p^{1/p^r}, \alpha\right)\right]\left[K\left(p^{1/p^r},\alpha\right): K(p^{1/p^r})\right],
$$
 $[K(p^{1/p^r},\alpha):K(p^{1/p^r})]$ divides $s$. Hence we obtain $[K(p^{1/p^r},\alpha):K(p^{1/p^r})]=\mathrm{gcd}(s, p^r)=1$. This shows  $K(p^{1/p^r})=K(\alpha)$ because $[K(p^{1/p^r}):K]=[K(\alpha):K ]$. This is a contradiction, and hence, $L_1$ and $L_2$ are not isomorphic.
\end{proof}

\begin{proposition}\label{wild_lower_bound}
Let $p$ be a prime number and let $e$ be a positive integer divisible by $p$.
 Then the lifting number for $\CC_{p,e}$ is bigger than $e$.
\end{proposition}

\begin{proof}
By Lemma \ref{wild distinct}, there are two  totally ramified extensions $L_1$ and $L_2$ of degree $e$ over $\BQ_p$ such that there is no isomorphism over $\BQ_p$ between $L_1$ and $L_2$.
 If $\CC_{p,e}$ is $e$-liftable, $L_1$ and $L_2$ are isomorphic over $\BQ_p$ by Proposition \ref{tame optimal pro} and it is a contradiction. Hence, the lifting number for $\CC_{p,e}$ is bigger than $e$.
\end{proof}

\begin{corollary}\label{lifting number >e}
The lifting number for $\CC_{p,e}$ is bigger than $e$ whenever $e>1$.
\end{corollary}
\noindent
Although we have  the lower bound $e+1$ and  the upper bound $e+e\nu(e)+1$ of the lifting number for $\CC_{p,e}$, we  have no clue to calculate the lifting number explicitly for the wild case.

\begin{question}
What is the lifting number for the wild case\textup{?}
\end{question}

\section{Ax-Kochen-Ershov principle for finitely ramified valued fields}\label{AKE_Basarab}

Our main goal in this section is to strengthen Basarab's result on relative completeness for finitely ramified henselian valued fields of mixed characteristic with perfect residue fields.
{\bf In this section, we drop the restriction that a valuation group is $\BZ$ so that a valuation group can be an arbitrary ordered abelian group.}
Recall that for a valued field $(K,R,\nu, \Gamma)$, $e_{\nu}(x)$ is the number of the positive elements of $\Gamma$ less than or equal to $\nu(x)$ for $x\in R$.

\begin{remark}\label{rmk:elementaryequivalence_same_ramification}
Let $(K_1,\nu_1)$ and $(K_2,\nu_2)$ be finitely ramified valued fields of mixed characteristic $(0,p)$. Suppose $R_{1,n}\equiv R_{2,n}$ for some $n> \min\{e_{\nu_1}(p),e_{\nu_2}(p)\}$, where $R_{1,(n)}$ and $R_{2,(n)}$ are the $n$-th residue rings of $K_1$ and $K_2$ respectively. Then, $e_{\nu_1}(p)=e_{\nu_2}(p)$.
\end{remark}
\begin{proof}
Without loss of generality, we may assume that $e_1:=e_{\nu_1}(p)\le e_2:=e_{\nu_2}(p)$. By the Keisler-Shelah isomorphism theorem, we may assume that $R_{1,(n)}$ and $R_{2,(n)}$ are isomorphic. Since $n> e_1$, we have that $pR_{1,(n)}=\bar \fm_1^{e_1}\neq 0$ where $\bar \fm_1$ is the maximal ideals of $R_{1,(n)}$. Since $R_{1,(n)}$ and $R_{2,(n)}$ are isomorphic, $0\neq pR_{2,(n)}=\bar \fm_2^{e_2}$, where $\bar \fm_2$ is the maximal ideal of $R_{2,(n)}$, and $e_1=e_2$.
\end{proof}

\begin{theorem}\label{fined_AKE_Rn}
Let $(K_1,\nu_1,\Gamma_1)$ and $(K_2,\nu_2,\Gamma_2)$ be finitely ramified henselian valued fields of mixed characteristic $(0,p)$ with perfect residue fields. Let $n_0>e_{\nu_2}(p)(1+e_{\nu_1}(e_{\nu_1}(p)))$. Then, the following are equivalent:
		\begin{enumerate}
			\item $K_1\equiv K_2$;
			\item $\Gamma_1\equiv \Gamma_2$ and $R_{1,{(n)}}\equiv R_{2,{(n)}}$ for each $n\ge 1$; and
			\item $\Gamma_1\equiv \Gamma_2$ and $R_{1,(n_0)}\equiv R_{2,(n_0)}$.
		\end{enumerate}
\end{theorem}
\begin{proof}
It is easy to check $(1)\Rightarrow(2)\Rightarrow(3)$. We show $(3)\Rightarrow (1)$.
Suppose $R_{1,(n_0)}\equiv R_{2,(n_0)}$ and $\Gamma_1\equiv \Gamma_2$.
By Remark \ref{iso_Rn_Gamma}, we may assume that $R_{1,(n_0)}\cong R_{2,(n_0)}$ and $\Gamma_1\cong \Gamma_2$, and that $(K_1,\nu_1,\Gamma_1)$ and $(K_2,\nu_2,\Gamma_2)$ are $\aleph_1$-saturated.
Consider the coarse valuations $\dnu_1$ and $\dnu_2$ of $\nu_1$ and $\nu_2$ respectively and the valued fields $(K_1,\dnu_1,\Gamma_1/\Gamma^{\circ}_1)$ and $(K_2,\dnu_2,\Gamma_2/\Gamma^{\circ}_2)$, where $\Gamma^{\circ}_i$ is the convex subgroup of $\Gamma_i$ generated by the minimum positive element in $\Gamma_i$ for $i=1,2$.
Since $(K_1,\nu_1)$ and $(K_2,\nu_2)$ are $\aleph_1$-saturated, by Remark \ref{coarse_valuation}.(4), the core fields $(\cfd_1,\cnu_1)$ and $(\cfd_2,\cnu_2)$ are complete discrete valued fields, where $\cnu_1$ and $\cnu_2$ are the valuations induced from $\nu_1$ and $\nu_2$ respectively.
Since the $n_0$-th residue rings of $(K_1,\nu_1)$ and $(K_2,\nu_2)$ are isomorphic, by Remark \ref{coarse_valuation}.(2), the $n_0$-th residue rings of $(\cfd_1,\cnu_1)$ and $(\cfd_2,\cnu_2)$ are isomorphic.

By Theorem \ref{main_theorem_construction}, $\cfd_1$ and $\cfd_2$ are isomorphic. Since $\Gamma_1\cong \Gamma_2$, $\Gamma_1/\Gamma^{\circ}_1\cong \Gamma_2/\Gamma^{\circ}_2$. Furthermore, $(K_1,\dnu_1,(K_1^{\circ},\nu_1^{\circ}))\equiv (K_2,\dnu_2,(K_2^{\circ},\nu_2^{\circ}))$ because Fact \ref{AKE} holds after adding structure on residue fields. To get that $(K_1,\nu_1)\equiv(K_2,\nu_2)$, it is enough to show that $(K_1,R_{\nu_1})\equiv (K_2,R_{\nu_2})$ in the ring language with a unary predicate. By Remark \ref{coarse_valuation}.(1), the valuation rings $R_{\nu_1}$ and $R_{\nu_2}$ are definable by the same formula in $(K_1,\dnu_1,(K_1^{\circ},\nu_1^{\circ}))$ and $(K_2,\dnu_2,(K_2^{\circ},\nu_2^{\circ}))$ so that $(K_1,R_{\nu_1})\equiv (K_2,R_{\nu_2})$.
\end{proof}

We give several corollaries of Theorem \ref{fined_AKE_Rn}. First, we improve the result in \cite{B2} on a decidability of finitely ramified henselian valued fields in the case of perfect residue field.
\begin{corollary}\label{decidability_ramified_valuedfield}
Let $(K,\nu,\Gamma)$ be a finitely ramified henselian valued field of mixed characteristic with a perfect residue field. Let $n_0>e_{\nu}(p)(1+e_{\nu}(e_{\nu}(p))$.
Let $\Th(K,\nu)$ be the theory of $(K,\nu)$, $\Th(\Gamma)$  the theory of $\Gamma$, and $\Th(R_{(n)})$ the theory of $R_{(n)}$.
The following are equivalent:
\begin{enumerate}
	\item $\Th(K,\nu)$ is decidable.
	\item $\Th(\Gamma)$ is decidable, and $\Th(R_{(n)})$ is decidable for each $n\ge 1$.
	\item $\Th(\Gamma)$ is decidable, and $\Th(R_{(n_0)})$ is decidable.
\end{enumerate}
Note that the lower bound of $n_0$ depends only on $e$ and $p$.
\end{corollary}
\begin{proof}
(1) $\Leftrightarrow$ (2) This was already given by Basarab in \cite{B2}.

(1) $\Leftrightarrow$ (3) Let $e(:=e_{\nu}(p))$ be the ramification index of $(K,\nu)$. Consider the following theory $T_{p,e}$ consisting of the following statements, which can be expressed by the first order logic;
\begin{itemize}
	\item $(K,\nu)$ is a henselian valued field of characteristic zero;
	\item $\Gamma$ is an abelian ordered group having the minimum positive element;
	\item $k$ is a perfect field of characteristic $p>0$;
	\item $(K,\nu)$ has the ramification index $e$.
\end{itemize}
By Theorem \ref{fined_AKE_Rn}, the theory $T_{p,e}\cup \Th(\Gamma)\cup \Th(R_{(n_0)})$ is complete. Thus $\Th(K,\nu)$ is decidable if and only if $\Th(\Gamma)$ and $\Th(R_{(n_0)})$ are decidable.
\end{proof}

Next we recall the following definition introduced in \cite{B1}:
\begin{definition}\label{elementary_bound}\cite{B1}
Let $T$ be the theory of a finitely ramified henselian valued field $(K,\nu,\Gamma)$ of mixed characteristic. Let $\lambda(T)\in\BN\cup\{\infty\}$ be defined as the smallest positive integer $n$ (if such a number exists) such that for every finitely ramified henselian valued field $(K',\nu',\Gamma')$ of mixed characteristic having the same ramification index of $(K,\nu,\Gamma)$, the following are equivalent:
\begin{enumerate}
	\item $(K',\nu',\Gamma')\models T$.
	\item $\Gamma\equiv \Gamma'$ and the $n$-th residue rings of $(K,\nu)$ and $(K',\nu')$ are elementarily equivalent.
\end{enumerate}
Otherwise, $\lambda(T)=\infty$.
\end{definition}

\noindent
Basarab in \cite{B1} showed that $\lambda(T)$ is finite if $T$ is the theory of a local field of mixed characteristic. In general, for the perfect residue field case, we prove that Basarab's invariant $\lambda(T)$ is always finite and  smaller than or equal to the lifting number.
\begin{corollary}\label{basarab_question_perfectfields}
Let $(K,\nu)$ be a finitely ramified henselian valued field of mixed characteristic $(0,p)$ having finite ramification index $e=e_{\nu}(p)$ with a perfect residue field. Let $T$ be the theory of $(K,\nu)$. Then
\begin{enumerate}
\item
 $\lambda(T)$ is smaller than or equal to the lifting number for $\CC_{p,e}$.
 \item
 $\lambda(T)\le e_{\nu}(p)(1+e_{\nu}(e_{\nu}(p))+1$.
 \end{enumerate}
\end{corollary}
\noindent Next, we compute explicitly $\lambda(T)$ for the theories $T$ of some tamely ramified valued fields. We say that an abelian group $G$ is $e$-divisible when the multiplication by $e$ map, $e:G\longrightarrow G$ is surjective. We denote the unit group of a ring $R$ by $R^{\times}$.

\begin{lemma}\label{e-divisible}
Let $(K, W(k),\fm,  k)$ be an unramified  complete discrete valued field of mixed characteristic $(0,p)$ with a perfect residue field. Suppose that $k^{\times}$ is $e$-divisible for a positive integer $e$ prime to $p$.
 \begin{enumerate}
 \item
If $\zeta_e$ is contained in $W(k)$, then there exists a unique totally ramified extension $L$ of degree $e$ over $K$.
\item
If $\zeta_e$ is not contained in $W(k)$, then there exists a unique totally ramified extension $L$ of degree $e$ over $K$ up to $K$-isomorphism.
\end{enumerate}
\end{lemma}

\begin{proof}
Let $S$ be the set of  Teichm\"{u}ller representatives of $W(k)$. By Hensel's lemma, $1+\fm$ is $e$-divisible, and so is $W(k)^{\times}=S\setminus \{0\} \times (1+\fm)$ because $k^{\times}\cong S\setminus \{0\}$ is $e$-divisible.

For a totally tamely ramified extension $L$ of degree $e$ over $K$, there is $u$ in $W(k)^{\times}$ such that $L=K(\sqrt[e]{pu})$ by the theory of tamely ramified extensions (c.f. \cite[Chapter 2]{L}). Since $W(k)^{\times}$ is $e$-divisible, there is $v$ in $W(k)^{\times}$ such that $v^e=u$. Hence, $\sqrt[e]{pu}= \sqrt[e]{p}v \zeta_e^i$ for some $i$. This shows that $L=K(\sqrt[e]{pu})=K(\sqrt[e]{p} \zeta_e^i)$ is isomorphic to $K(\sqrt[e]{p})$ over $K$ because the irreducible polynomial of $\sqrt[e]{p}$ over $K$ is $x^e-p$.
Furthermore, $L=K(\sqrt[e]{p})$ when $\zeta_e $ is contained in $W(k)$.
%
\end{proof}

\noindent

\begin{proposition}\label{lambda(T)_tame}
Let $(K,\nu,\Gamma,k)$ be a finitely tamely ramified henselian valued field of mixed characteristic $(0,p)$ with a perfect residue field. Let $e\ge 2$ be the ramification index of $(K,\nu)$. Let $T$ be the theory of $(K,\nu)$.
\begin{enumerate}
	\item If $k^{\times}$ is $e$-divisible, then $\lambda(T)=1$.
	\item If there is a prime divisor $l$ of $e$ such that $\zeta_{l^n}\in k^{\times}$ and $\zeta_{l^{n+1}}\notin k^{\times}$ for some $n$, then $\lambda(T)=e+1$.
\end{enumerate}
\end{proposition}
\begin{proof}
(1)
 Suppose $k^{\times}$ is $e$-divisible. Let $(K',\nu',\Gamma',k')$ be a henselian valued field of mixed characteristic having ramification index $e$. Suppose $k\equiv k'$ and $\Gamma\equiv \Gamma'$. By Remark \ref{iso_Rn_Gamma}, we may assume that $k\cong k'$, $\Gamma\cong \Gamma'$, and both $K$ and $K'$ are $\aleph_1$-saturated. Consider the core fields $(K^{\circ},\nu^{\circ},k^{\circ})$ and $( (K')^{\circ} ,(\nu')^{\circ},(k')^{\circ} )$ of $(K,\nu)$ and $(K',\nu')$ respectively. Since $k^{\times}$ is $e$-divisible, so is $(k^{\circ})^{\times}$. Then by Lemma \ref{e-divisible}, $(K^{\circ},\nu^{\circ}) \cong ((K')^{\circ},(\nu')^{\circ})$. By the proof of Theorem \ref{fined_AKE_Rn}, we have $(K,\nu)\equiv (K',\nu')$. Thus $\lambda(T)=1$.
\smallskip

(2)
 Suppose there is a prime divisor $l$ of $e$ and a natural number $n$ such that $\zeta_{l^n}\in k^{\times}$ and $\zeta_{l^{n+1}}\notin k^{\times}$. Let $T_{p,e}$ be the theory introduced in the proof of Corollary \ref{decidability_ramified_valuedfield}. Set $T_0=T_{p,e}\cup \Th(R_{e})$. Consider the following theories:
\begin{itemize}
	\item $T_1=T_0\cup \{\exists x( x^e-p=0)\}$;
	\item $T_2=T_0\cup\left\{\exists xy\big((x^e-py=0)\wedge \Phi_{l^n}(y)=0\big)\right\}$,
\end{itemize}
 where $\Phi_{l^n}(X)\in \BZ[X]$ is the $l^n$-th cyclotomic polynomial. By the proof of Lemma \ref{tame distinct}, we have
\begin{itemize}
	\item $T_1\cup T_2$ is inconsistent;
	\item $T_1$ and $T_2$ are consistent.
\end{itemize}
So, there are at least two different complete theories containing $T_0$, and we have $\lambda(T)\ge e+1$. By Corollary \ref{basarab_question_perfectfields}, we conclude that $\lambda(T)= e+1$.
\end{proof}
\noindent For some wild cases, we have a lower bound for $\lambda(T)$.
\begin{proposition}\label{lower_bound_lambda(T)_wild}
Let $p$ be a prime number and let $e$ be a positive integer divisible by $p$.
Let $(K,\nu,\Gamma,k)$ be a finitely ramified henselian valued field of mixed characteristic $(0,p)$ with a perfect residue field having the ramification index $e\ge 2$.
Then $\lambda(T)\ge e+1$ for the theory $T$ of $(K,\nu)$.
\end{proposition}
\begin{proof}
The proof is similar to the proof of Proposition \ref{lambda(T)_tame}.
Let $T_{p,e}$ and $T_0$ be the theory introduced in the proof of Proposition \ref{lambda(T)_tame}.
We write $e=sp^r$ for positive integers $s$ and $r$ where $s$ is prime to $p$.
Let $\alpha\in \BQ^{alg}$ be  as in the proof of Lemma \ref{wild distinct}.
In particular, $\alpha$ is a uniformizer of $M_r$ corresponding to the place above $p$ where
$M_r=\BQ(\alpha)$ is the $r$-th subfield of the cyclotomic $\BZ_p$-extension $\BQ_{\infty}$ of degree $p^r$ over $\BQ$.
Let $f(X)$ be the minimal polynomial of $\alpha$ over $\BQ$. Consider the following theories:
\begin{itemize}
	\item $T_1=T_0\cup\{\exists x(x^e-p=0)\}$;
	\item $T_2=T_0\cup\{\exists x(x^s-p=0),\ \exists x(f(x)=0)\}$.
\end{itemize}
By the proof of Lemma \ref{wild distinct}, we have
\begin{itemize}
	\item $T_1\cup T_2$ is not consistent;
	\item $T_1$ and $T_2$ are consistent.
\end{itemize}
So, there are at least two different complete theories containing $T_0$, we have $\lambda(T)\ge e+1$.
\end{proof}
\noindent
We list some special cases of Proposition \ref{lambda(T)_tame} and Proposition \ref{lower_bound_lambda(T)_wild} (see Corollary \ref{cor:explicit_value_basarab_ivariants}).
For a positive integer $s$, we say that $s^{\infty}$ divides $[k:\BF_p]$ if there is a subfield $k_n$ of $k$ such that $[k_n:\BF_p]$ is finite and $s^n$ divides $[k_n:\BF_p]$ for each $n\ge 1$.
For $m\ge 1$, let $\mu_m$ be the group generated by $\zeta_m$ and let $\mu_{m^{\infty}}=\bigcup_{n\ge 1}\mu_{m^n}$.
\begin{remark}\label{rem:criterion_for_e-divisibility}
Let $k$ be an algebraic extension of $\mathbb{F}_p$. Let $e>1$ be coprime to $p$, and let $s$ be the order of the group $\mu_e\cap k^{\times}$. Suppose $s^{\infty}$ divides $[k:\mathbb F_p]$. Then, $k^{\times}$ is $e$-divisible.
\end{remark}
\begin{proof}
Note that $(k^{alg})^{\times}\cong \oplus \mu_{q^{\infty}}$ where $q$ runs through all primes not equal to $p$. To show that $k^{\times}$ is $e$-divisible, it is enough to show that $k^{\times}$ is $r$-divisible for each prime factor $r$ of $e$. \\

Case $r\nmid s$. $k^{\times}$ is contained in $\oplus_{q\neq p,r} \mu_{q^{\infty}}$. Since $\mu_{q^n}$ is $r$-divisible for each $q\neq r$, $k^{\times}$ is $r$-divisible.\\

Case $r \mid s$. Note that $r^{\infty}$ divides $[k:\mathbb F_p]$ because $s^{\infty}$ divides $[k:\mathbb F_p]$. It is enough to show that $\mu_{r^{\infty}}\subset k^{\times}$.
Clearly, we have that $\zeta_r\in k$. By Kummer theory, for any positive integer $n$, we have $[\mathbb F_p(\zeta_{r^{n+1}}):\mathbb F_p]=r^{d_n}[\mathbb F_p(\zeta_r):\mathbb F_p]$ for some $d_n\le n$. Since $r^{\infty}$ divides $[k:\mathbb F_p]$, there is a subfield $k_{r,n}$ of $k$ with $[k_{r,n}:\mathbb F_p]=r^n$ so that $[k_{r,n}(\zeta_{r}):\mathbb F_p]=r^n[\mathbb F_p(\zeta_r):\mathbb F_p]$. So, $\mathbb F_p(\zeta_{r^{n+1}})\subset k_{r,n}(\zeta_{r})\subset k$. Therefore, we conclude that $\mu_{r^{\infty}}\subset k$.
\end{proof}

\begin{corollary}\label{cor:explicit_value_basarab_ivariants}

Let $(K,\nu,\Gamma,k)$ be a finitely ramified henselian valued field of mixed characteristic $(0,p)$ with a perfect residue field. Let $e$ be the ramification index of $K$ and let $s$ be the order of the group $\mu_e \cap k^{\times}$ where $\mu_e$ is the group generated by $\zeta_e$. For the theory $T$ of $(K,\nu)$,

Case $p\nmid e$.
\begin{itemize}
\item
$\lambda(T)=1$ when $k=k^{alg}$;
\item
$\lambda(T)=1$ when $K$ is a subfield of $\BC_p$ and
$s^\infty$ divides $[k:\BF_p]$;
\item
$\lambda(T)=e+1$ when $K$ is a subfield of $\BC_p$ and $s^\infty$ does not divide $[k:\BF_p]$.
\end{itemize}

Case $p|e$.
\begin{itemize}
\item $\lambda(T)\ge e+1$ when $K$ is a subfield of $\BC_p$.
\end{itemize}

\end{corollary}

\noindent Propositon \ref{lambda(T)_tame}.(1) shows that Basarab's invariant $\lambda(T)$ can be strictly smaller than the bound in Corollary \ref{basarab_question_perfectfields} for the tame case. In the following example, the same thing can happen for the wild case.
\begin{example}

Let $(K,R,\nu)=(\BQ_3(\sqrt[3]3), \BZ_3[\sqrt[3]3], \nu)$, $f(x)=x^3-3$ and $\alpha_1=\sqrt[3]3$, $\alpha_2= \sqrt[3]3\zeta_3$, and $\alpha_3=\sqrt[3]3\zeta_3 ^2 $.
 Since  $f(x)=(x-\sqrt[3]{3})(x-\sqrt[3]{3}\zeta_3)(x-\sqrt[3]{3}\zeta_3 ^2)=(x-\alpha_1)(x-\alpha_2)(x-\alpha_3)$ and $[\BQ_3(\sqrt[3]{3}, \zeta_3 ) :\BQ_3(\sqrt[3]{3})]=2$,
$$
\frac{x^3-3}{x-\sqrt[3]{3}}=\left(x-\sqrt[3]{3}\zeta_3\right)\left(x-\sqrt[3]{3}\zeta_3 ^2\right)=\left(x-\alpha_2\right)\left(x-\alpha_3\right)
$$
is irreducible over $\BQ_3(\sqrt[3]{3})$, that is, $\alpha_2$ and $\alpha_3$ are conjugate each other over $\BQ_3(\sqrt[3]3)$.
It follows that $\widetilde{\nu}(\alpha_1-\alpha_2)=\widetilde{\nu}(\alpha_1-\alpha_3)$.
By {\emph Fact \ref{different bound}},
$$
\widetilde{\nu}(f'(\alpha_1))=\widetilde{\nu}((\alpha_1-\alpha_2)(\alpha_1-\alpha_3))=2 \widetilde{\nu}(\alpha_1-\alpha_2)\le \frac{\nu(3)-1+\nu\big(\nu(3)\big)}{\nu(3)}.
$$
 Hence we have the following bound
\begin{align*}
M(R)& = \max\left\{\wi{\nu}\big(\alpha_1- \alpha_j \big): j\neq 1 \right\}\\
&=\widetilde{\nu}(\alpha_1-\alpha_2)\ =\ \widetilde{\nu}(\alpha_1-\alpha_3)\ =\ \frac{\widetilde{\nu}(f'(\alpha_1))}{2}\\
&\le \frac{\nu(3)-1+\nu\big(\nu(3)\big)}{2\nu(3)}\ =\ \frac{3-1+\nu(3)}{6}\\
&=\frac{5}{6}.
\end{align*}
 So we have
$$
M(R)\nu(3)^2\le\frac{5}{6} 3^2 =\frac{15}{2} \le8 < \nu(3)+\nu(3)\nu(\nu(3))=3+3\nu(3)=12.
$$
Thus, \emph{Theorem \ref{main_theorem_construction}} shows that Basarab's invariant $\lambda(T)$ for $K$ is smaller than or equal to $8$, which is strictly smaller than $\nu(3)(1+\nu(\nu(3)))+1=12$.
\end{example}

\end{document}